\RequirePackage{fix-cm}
\documentclass[smallextended]{svjour3}       
\smartqed  
\usepackage{graphicx}
\usepackage{geometry}
\usepackage{color}
\newcommand{\tr}{^{\sf T}}
\newcommand{\m}[1]{{\bf{#1}}}

\newcommand{\C}[1]{{\cal {#1}}}

\begin{document}

\title{ Generalized Symmetric ADMM for Separable Convex Optimization
\thanks{The work was supported by the National Science Foundation of China under grants 11671318 and 11571178,
the Science Foundation of Fujian province of China  under grant 2016J01028,
and the National Science Foundation of U.S.A. under grant 1522654.}
}

\author{Jianchao Bai   \and Jicheng Li  \and Fengmin Xu \and Hongchao Zhang
}


\institute{
    Jianchao Bai
    \at School of Mathematics and Statistics, Xi'an Jiaotong University, Xi'an 710049, P.R. China\\
    \email{{\tt bjc1987@163.com}}
\and
    Jicheng Li
    \at School of Mathematics and Statistics, Xi'an Jiaotong University, Xi'an 710049, P.R. China\\
    \email{{\tt jcli@mail.xjtu.edu.cn}}
\and
    Fengmin Xu
    \at School of Economics and Finance, Xi'an Jiaotong University, Xi'an 710049, P.R. China\\
    \email{{\tt fengminxu@mail.xjtu.edu.cn}}
\and
    Hongchao Zhang \at Department of Mathematics,
    Louisiana State University, Baton Rouge, LA 70803-4918, U.S.A.\\
    \email{{\tt hozhang@math.lsu.edu}}
}

\date{Received: date / Accepted: date}

\maketitle

\begin{abstract}
The Alternating Direction Method of Multipliers (ADMM) has been proved to be effective for solving separable convex optimization subject to linear constraints. In this paper, we propose a Generalized Symmetric ADMM (GS-ADMM), which updates the Lagrange multiplier twice with suitable stepsizes, to solve the multi-block separable convex programming. This GS-ADMM partitions the data into two group variables so that one group consists of $p$ block variables while the other has $q$ block variables, where $p \ge 1$ and $q \ge 1$ are two integers. The two grouped variables are updated in a {\it Gauss-Seidel}  scheme, while the variables within each group are updated in a {\it Jacobi} scheme, which would make it very attractive for a big data setting.  By adding proper proximal terms to the subproblems, we specify the domain of the stepsizes to guarantee that GS-ADMM is globally convergent with a worst-case  $\C{O}(1/t)$ ergodic convergence rate. It turns out that our convergence domain of the stepsizes is significantly larger than other convergence domains in the literature. Hence, the GS-ADMM is more flexible and attractive  on choosing and using larger stepsizes of the dual variable. Besides, two special cases of GS-ADMM, which allows using zero penalty terms, are also discussed and analyzed.
Compared with several state-of-the-art methods, preliminary numerical
experiments on solving a sparse matrix minimization problem in the statistical learning show that our
proposed method is effective and promising.

\end{abstract}

\section{Introduction}
We  consider the following grouped multi-block separable convex programming problem
\begin{equation} \label{Prob}
\begin{array}{lll}
\min  &  \sum\limits_{i=1}^{p}f_i(x_i)+ \sum\limits_{j=1}^{q}g_j(y_j) \\
\textrm{s.t. } &   \sum\limits_{i=1}^{p}A_i x_i+\sum\limits_{j=1}^{q}B_jy_j=c,\\
     &   x_i\in \mathcal{X}_i, \; i=1, \cdots, p, \\
     &   y_j\in \mathcal{Y}_j, \; j =1, \cdots, q, \\
\end{array}
\end{equation}
 where $f_i(x_i):\mathcal{R}^{m_i}\rightarrow\mathcal{R},\ g_j(y_j):\mathcal{R}^{d_j}\rightarrow\mathcal{R}$ are closed and proper convex functions
(possibly nonsmooth); $A_i\in\mathcal{R}^{n\times m_i}, B_j\in\mathcal{R}^{n\times d_j}$ and $ c\in\mathcal{R}^{n}$ are given  matrices and vectors, respectively;  $\mathcal{X}_i\subset \mathcal{R}^{m_i}$ and $\mathcal{Y}_j\subset \mathcal{R}^{d_j} $   are closed convex sets; $p \ge 1$ and $q \ge 1$ are two integers. Throughout this paper, we assume that the solution set of the problem (\ref{Prob}) is nonempty and all the matrices $A_i$, $i=1,\cdots, p$, and $B_j$, $j=1,\cdots,q$, have full column rank.
And in the following, we denote $\mathcal{A}=\left(A_1, \cdots,A_p\right),\mathcal{B}=\left(B_1, \cdots,B_q\right)$,
$\m{x} = (x_1 \tr, \cdots, x_p \tr) \tr$, $\m{y} = (y_1 \tr, \cdots, y_q \tr) \tr$, $\C{X} = \C{X}_1 \times \C{X}_2 \times \cdots \C{X}_p $,
$\C{Y} = \C{Y}_1 \times \C{Y}_2 \times \cdots \C{Y}_q $ and $\mathcal{M}=\mathcal{X} \times \mathcal{Y} \times \mathcal{R}^n$.

In the last few years, the problem (\ref{Prob}) has been extensively investigated due to its wide applications in different fields, such as the sparse inverse covariance estimation problem \cite{RothmanBickel2008}
in finance and statistics, the model updating problem \cite{DongYuTian2015} in the design of vibration structural dynamic system and bridges,
 the low rank and sparse representations \cite{LiuLiBai2017} in image processing and so forth. One standard way to solve the problem (\ref{Prob}) is the  classical Augmented Lagrangian Method (ALM) \cite{Hestenes1969}, which minimizes the following augmented Lagrangian function
\[
 \mathcal{L}_\beta\left(\m{x}, \m{y},\lambda\right)=L\left(\m{x}, \m{y},\lambda\right)+\frac{\beta}{2} \| \C{A} \m{x} + \C{B} \m{y} -c\|^2,
\]
where $\beta>0$ is a penalty parameter for the equality constraint and
\begin{equation} \label{lagrange}
L\left(\m{x}, \m{y},\lambda\right)=\sum\limits_{i=1}^{p}f_i(x_i)+ \sum\limits_{j=1}^{q}g_j(y_j) -  \langle \lambda, \C{A} \m{x} + \C{B} \m{y} -c \rangle
\end{equation}
is the Lagrangian function of the problem (\ref{Prob}) with the Lagrange multiplier $\lambda\in \mathcal{R}^{n}$.  Then,  the ALM procedure for solving (\ref{Prob}) can
be described as  follows:
\[
\left \{\begin{array}{l}
\left(\m{x}^{k+1},\m{y}^{k+1}\right)=\arg\min\left\{ \mathcal{L}_\beta\left(\m{x},\m{y},
\lambda^k\right)|\ \m{x} \in\mathcal{X}, \m{y} \in \mathcal{Y}\right\},\\
\lambda^{k+1}=\lambda^{k}-\beta  (\C{A}\m{x}^{k+1} + \C{B}\m{y}^{k+1} -c).
\end{array}\right.
\]
However, ALM does not make full use of the separable structure of the
objective function of (\ref{Prob}) and hence, could not take advantage of
the special properties of the component objective functions $f_i$ and $g_j$
in (\ref{Prob}).
As a result, in many recent real applications involving big data, solving the subproblems of ALM becomes very expensive.

One effective approach to overcome such difficulty is the Alternating Direction
Method of Multipliers (ADMM), which
was originally proposed in \cite{GlowinskiMarrocco1975} and  could  be regarded as a splitting version of ALM. At each iteration, ADMM first sequentially
optimize over one block variable while fixing all the other block variables,
and then follows by  updating the Lagrange multiplier.
 A natural extension of ADMM for solving the multi-block case problem (\ref{Prob}) takes the following
iterations:
\begin{equation}\label{gs-admm}
  \left \{\begin{array}{lll}
\textrm{For}\ i=1,2,\cdots,p,\\
\qquad x_{i}^{k+1}=\arg\min\limits_{x_{i}\in\mathcal{X}_{i}} \mathcal{L}_\beta\left(x_1^{k+1},\cdots,x_{i-1}^{k+1}, x_i,x_{i+1}^k,\cdots, x_{p}^k,\m{y}^k,\lambda^k\right),\\
\textrm{For}\ j=1,2,\cdots,q,\\
\qquad y_{j}^{k+1}=\arg\min\limits_{y_{j}\in\mathcal{Y}_{j}} \mathcal{L}_\beta\left(\m{x}^{k+1},y_1^{k+1},\cdots,y_{j-1}^{k+1}, y_j,y_{j+1}^k,\cdots, y_{q}^k,\lambda^k\right),\\
\lambda^{k+1}=\lambda^k-\beta\left(\mathcal{A}\m{x}^{k+1}+\mathcal{B}\m{y}^{k+1}-c\right).
\end{array}\right.
\end{equation}
Obviously, the scheme (\ref{gs-admm}) is  a serial algorithm which uses the newest information of the variables at each
iteration. Although the above scheme was proved to be convergent for the two-block,
i.e., $p=q=1$, separable convex minimization (see \cite{HeYuan2012}), as shown in
\cite{Chen2016}, the direct extension of ADMM (\ref{gs-admm}) for the multi-block
case, i.e., $p+q \ge 3$, without proper modifications is not necessarily convergent.
Another natural extension of ADMM is to use the Jacobian fashion,
 where the variables  are updated simultaneously after each iteration, that is,
\begin{equation}\label{ja-admm}
 \left \{\begin{array}{lll}
\textrm{For}\ i=1,2,\cdots,p,\\
\qquad x_{i}^{k+1}=\arg\min\limits_{x_{i}\in\mathcal{X}_{i}} \mathcal{L}_\beta\left(x_1^k,\cdots,x_{i-1}^k, x_i,x_{i+1}^k,\cdots, x_{p}^k,
\m{y}^k,\lambda^k\right),\\
\textrm{For}\ j=1,2,\cdots,q,\\
\qquad y_{j}^{k+1}=\arg\min\limits_{y_{j}\in\mathcal{Y}_{j}} \mathcal{L}_\beta\left(\m{x}^k,y_1^k,\cdots,y_{j-1}^k, y_j,y_{j+1}^k,\cdots, y_{q}^k,\lambda^k\right),\\
\lambda^{k+1}=\lambda^k-\beta\left(\mathcal{A}\m{x}^{k+1}+\mathcal{B}\m{y}^{k+1}-c\right).
\end{array}\right.
\end{equation}
As shown in \cite{HeHouYuan2015}, however, the Jacobian scheme (\ref{ja-admm})
is not necessarily convergent either.
To ensure the convergence,
He et al. \cite{HeTaoYuan2015} proposed a novel  ADMM-type splitting method that by adding certain proximal terms,
allowed some of the subproblems to be solved in parallel, i.e., in a Jacobian fashion.
And in \cite{HeTaoYuan2015}, some sparse low-rank models and image  painting problems were tested to
verify the efficiency of their method.

More recently, a Symmetric ADMM (S-ADMM) was proposed by He et al. \cite{HeMaYuan2016} for solving the two-block separable convex minimization, where the algorithm performs the following updating scheme:
\begin{equation}\label{he-admm}
\left \{\begin{array}{lll}
\m{x}^{k+1}=\arg\min\left\{ \mathcal{L}_\beta\left(\m{x},\m{y}^k,\lambda^k\right) \ | \ \m{x} \in\mathcal{X} \right\},\\
\lambda^{k+\frac{1}{2}}=\lambda^k- \tau \beta\left(\C{A} \m{x}^{k+1}+ \C{B} \m{y} ^{k}-c\right),\\
\m{y}^{k+1}=\arg\min\left\{ \mathcal{L}_\beta( \m{x}^{k+1},\m{y},\lambda^{k+\frac{1}{2}}) \ | \ \m{y} \in\mathcal{Y}\right\},\\
\lambda^{k+1}=\lambda^{k+\frac{1}{2}}-s\beta\left(\C{A} \m{x}^{k+1}+\C{B} \m{y}^{k+1}-c\right),
\end{array}\right.
\end{equation}
and the stepsizes $(\tau,s)$ were restricted into the domain
\begin{equation}\label{region-he-admm}
\mathcal{H}=\left\{(\tau,s) \ | \ s \in (0, (\sqrt{5} + 1)/2 ),  \
\tau+s >0, \ \tau \in (-1,1), \ |\tau|<1+s-s^2 \right\}
\end{equation}
in order to ensure its global convergence.
The main improvement of \cite{HeMaYuan2016} is that the  scheme (\ref{he-admm}) largely extends the domain of the
stepsizes $(\tau,s)$ of other ADMM-type methods  \cite{HeLiuWangYuan2014}. What's more, the numerical performance of
S-ADMM  on solving the widely used basis pursuit model and the total-variational image debarring model significantly
outperforms the original ADMM in both the CPU time and the number of iterations. Besides, Gu, et al.\cite{GuJiang2015} also studied a semi-proximal-based strictly contractive Peaceman-Rachford splitting method, that is (\ref{he-admm}) with two additional proximal penalty terms for the $x$ and $y$ update.
But their method has a nonsymmetric convergence domain of the stepsize and still focuses on the two-block case problem, which limits its applications for solving
large-scale problems with multiple block variables.

Mainly motivated by the work of \cite{HeTaoYuan2015,HeMaYuan2016,GuJiang2015}, we would like to generalize
S-ADMM with more wider convergence domain of the stepsizes
to tackle the multi-block separable convex programming model (\ref{Prob}), which more frequently appears in recent applications
involving big data \cite{Chandrasekaran2012,Ma2017}.
Our algorithm framework can be described as follows:
\begin{equation}\label{GS-ADMM}
(\textbf{GS-ADMM})\   \left\{\begin{array}{lll}
 \textrm{For}\ i=1,2,\cdots,p,\\
 \quad x_{i}^{k+1}=\arg\min\limits_{x_{i}\in\mathcal{X}_{i}} \mathcal{L}_\beta (x_1^k,\cdots, x_i,\cdots,x_p^k,\m{y}^k,\lambda^k)
+P_i^k(x_i), \\
\quad \textrm{where } P_i^k(x_i) =
\frac{\sigma_1\beta}{2}\left\|A_{i}(x_{i}-x_{i}^k)\right\|^2,\\
 \lambda^{k+\frac{1}{2}}=\lambda^k-\tau\beta(\mathcal{A} \m{x}^{k+1}+\mathcal{B}\m{y}^{k}-c),\\ \\
 \textrm{For}\ j=1,2,\cdots,q,\\
 \quad y_{j}^{k+1}=\arg\min\limits_{y_{j}\in\mathcal{Y}_{j}} \mathcal{L}_\beta(\m{x}^{k+1},y_1^k,\cdots, y_j,\cdots,y_q^k,\lambda^{k+\frac{1}{2}}) + Q_j^k(y_j), \\
 \quad \textrm{where } Q_j^k(y_j) =
\frac{\sigma_2\beta}{2}\left\|B_{j}(y_{j}-y_{j}^k)\right\|^2,\\
\lambda^{k+1}=\lambda^{k+\frac{1}{2}}-s\beta(\mathcal{A}\m{x}^{k+1}+\mathcal{B}\m{y}^{k+1}-c).
\end{array}\right.
\end{equation}
In the above Generalized Symmetric ADMM ({\bf GS-ADMM}), $\tau$ and $s$ are two stepsize parameters satisfying
\begin{equation}\label{setK}
  (\tau, s) \in \C{K} = \left\{ (\tau, s) \ | \   \tau + s >0, \ \tau \le 1,\ -\tau^2 - s^2 -\tau s + \tau + s + 1 >0 \right\},
\end{equation}
and $\sigma_1\in (p-1,+\infty),\sigma_2\in (q-1,+\infty)$ are two proximal parameters\footnote{ Note that these two parameters are strictly positive in (\ref{GS-ADMM}).  In Section 4,  however, we analyze two special cases of GS-ADMM allowing either $\sigma_1$ or $\sigma_2$ to be zero.}
for the regularization terms $P_i^k(\cdot)$ and $Q_j^k(\cdot)$.
 He and Yuan\cite{Heyuan2015} also investigated the above GS-ADMM (\ref{GS-ADMM}) but restricted  the stepsize $\tau=s\in(0,1)$,
which does not exploit the advantages of using flexible stepsizes given in (\ref{setK}) to improve its convergence.

\begin{figure}[htbp]\label{figK}
 \begin{minipage}{1\textwidth}
 \centering
\resizebox{10cm}{8.5cm}{\includegraphics{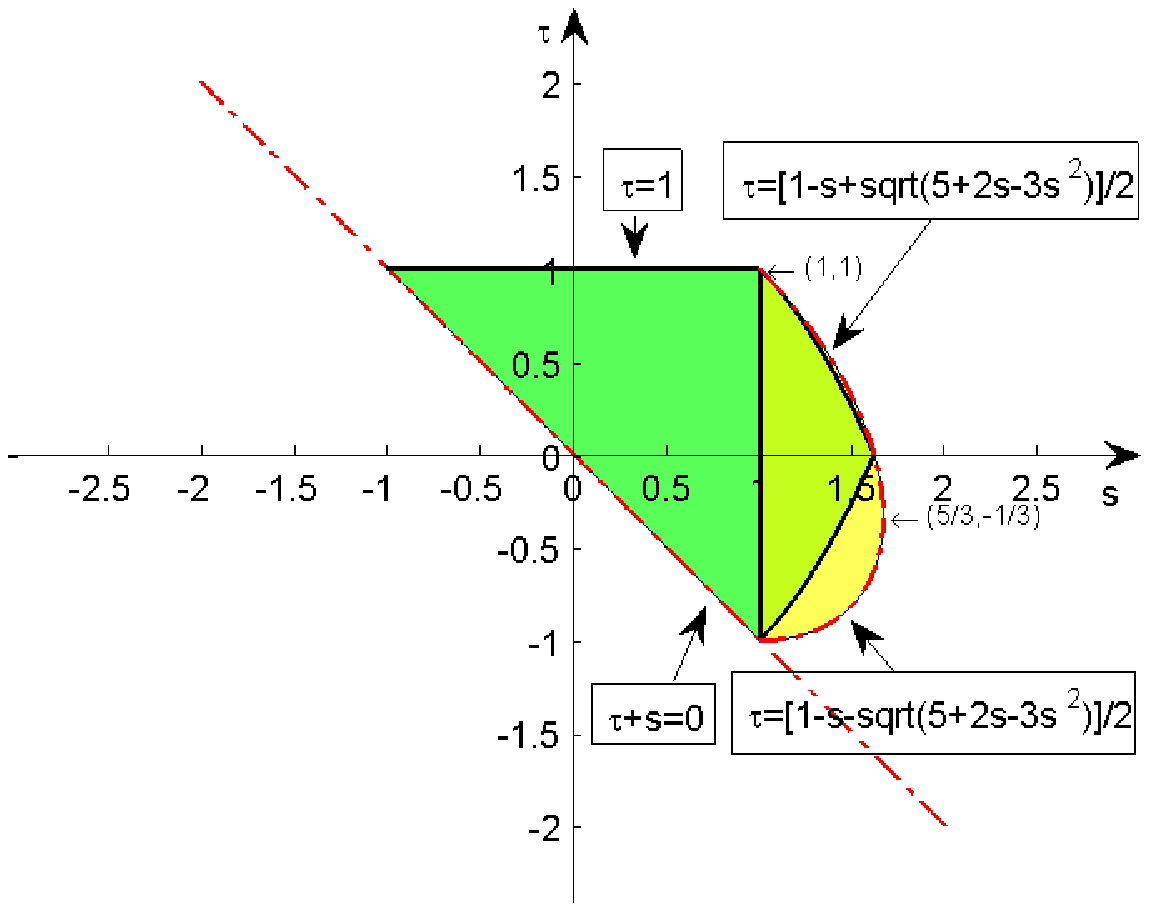}}
   \end{minipage}
\begin{center}
{\footnotesize Fig 1. Stepsize region $\mathcal{K}$ of GS-ADMM}
\end{center}
\end{figure}

Major contributions of this paper can be summarized as the following four aspects:
\begin{itemize}
\item Firstly, the new GS-ADMM  could deal with
the multi-block separable convex programming problem (\ref{Prob}), while the original S-ADMM in \cite{HeMaYuan2016} only
works for the two block case and may not be convenient for solving large-scale problems.
 In addition, the convergence domain $\mathcal{K}$ for the stepsizes $(\tau, s)$ in (\ref{setK}), shown in Fig. 1, is significantly larger than
the domain $\mathcal{H}$ given in  (\ref{region-he-admm}) and the convergence domain in \cite{GuJiang2015,Heyuan2015}.  For example, the stepsize $s$ can be arbitrarily close to $5/3$ when the stepsize $\tau$ is close to $-1/3$. Moreover, the above domain in (\ref{setK}) is later enlarged to a symmetric domain $\C{G}$ defined in (\ref{setG}), shown in Fig. 2.
 Numerical experiments in Sec. 5.2.1 also validate  that using more flexible and relatively larger stepsizes $(\tau,s)$ can
often improve the convergence speed of GS-ADMM.
On the other hand, we can see that when $\tau=0$, the stepsize $s$ can be chosen in the interval $(0, (\sqrt{5}+1)/2)$, which was
firstly suggested by Fortin and Glowinski in \cite{FortinGlowinski1983,Glowinski1984}.
\item Secondly, the global convergence of GS-ADMM as well as its  worst-case
 $\C{O}(1/t)$  ergodic convergence rate are established.
What's more, the total $p+q$ block variables are  partitioned into two  grouped  variables. While a {\it Gauss-Seidel} fashion is
taken between the two grouped variables, the block variables within each group are updated in a {\it Jacobi} scheme.
Hence,  parallel computing can be implemented for updating the variables within each group,  which could be critical in some
scenarios for problems involving big data.
\item Thirdly, we discuss two special cases of GS-ADMM, which is
 (\ref{GS-ADMM}) with $p\geq 1, q=1$ and $\sigma_2=0$ or with $p=1,q\geq 1$ and $\sigma_1=0$.
 These two special cases of GS-ADMM were not discussed in \cite{Heyuan2015} and in fact,
to the best of our knowledge, they have not been studied in the literature.
We show the convergence
domain of the stepsizes $(\tau, s)$ for these two cases is still $\C{K}$ defined in (\ref{setK})  that  is larger than
$\mathcal{H}$.
\item Finally, numerical experiments are performed on solving a sparse matrix
optimization problem arising from the statistical learning. We have investigated
the  effects  of the stepsizes $(\tau,s)$ and the penal parameter $\beta$
on the performance of GS-ADMM. And our numerical experiments demonstrate that by properly choosing
the parameters, GS-ADMM could perform significantly better than other
recently quite popular methods developed in  \cite{BaiLiLi2017,HeTaoYuan2015,HeXuYuan2016,WangSong2017}.
\end{itemize}

The paper is organized as follows. In Section 2, some preliminaries are given to reformulate the problem (\ref{Prob})
into a variational inequality and to interpret the GS-ADMM (\ref{GS-ADMM}) as a prediction-correction procedure. Section 3 investigates some properties of $\|\m{w}^k-\m{w}^*\|_H^2$ and  provides a lower bound of $\|\m{w}^{k}-\widetilde{\m{w}}^k\|_G^2$, where $H$ and $G$ are some particular
symmetric matrices. Then, we establish the global convergence of GS-ADMM and show its convergence rate in an ergodic sense. In Section 4,
we discuss two special cases of GS-ADMM, in which either the penalty parameters $\sigma_1$ or $\sigma_2$ is allowed to be zero. Some preliminary numerical
experiments are done in Section 5.
 We finally make some conclusions in Section 6.

\subsection{Notation}
Denoted by $\mathcal{R}, \mathcal{R}^n, \mathcal{R}^{m\times n}$ be the set
of  real numbers, the set of $n$ dimensional real column vectors and the set of $m\times n$ real matrices, respectively. For any $x, y\in \mathcal{R}^n$,
$\langle x, y\rangle=x \tr y$ denotes their inner product and
$\|x\|=\sqrt{\langle x, x\rangle}$ denotes the Euclidean norm of $x$, where  the superscript
$\tr$ is the transpose. Given a symmetric matrix $G$, we define $\|x\|_G^2 = x \tr G x$.
Note that with this convention, $\|x\|_G^2$ is not necessarily nonnegative unless
$G$ is a positive definite matrix ($\succeq\textbf{0}$). For convenience,
we  use $I$ and $\bf{0}$ to  stand respectively for the identity matrix and the zero matrix with
proper dimension throughout the context.

\section{Preliminaries}
In this section,  we first use a variational inequality to characterize the solution
set of the problem (\ref{Prob}). Then, we  analyze that GS-ADMM (\ref{GS-ADMM})
can be treated as a prediction-correction procedure  involving
a prediction step  and a correction step.

\subsection{Variational reformulation of (\ref{Prob})}
We begin with the following standard lemma whose proof can be found in \cite{HeMaYuan2016,He2016}.
\begin{lemma} \label{opt-1}
Let $f:\mathcal{R}^m\longrightarrow \mathcal{R}$ and
$h:\mathcal{R}^m\longrightarrow \mathcal{R}$ be two convex functions  defined on a closed convex
set $\Omega\subset \mathcal{R}^m$ and $h$ is differentiable. Suppose that the solution set
$\Omega^* = \arg\min\limits_{x\in\Omega}\{f(x)+h(x) \}$ is nonempty. Then, we have
\[
x^* \in \Omega^* \ \mbox{ if and only if } \ x^*\in \Omega, \  f(x)-f(x^*)
+\left\langle x-x^*, \nabla h(x^*)\right\rangle\geq 0, \forall x\in\Omega.
\]
\end{lemma}

It is well-known in optimization that a triple point $(\m{x}^*, \m{y}^*, \lambda^*) \in \C{M}$ is called the  saddle-point  of the
Lagrangian function  (\ref{lagrange}) if it satisfies
\[
L\left(\m{x}^*,\m{y}^*,\lambda\right)\leq L\left(\m{x}^*, \m{y}^*, \lambda^*\right)\leq
L\left(\m{x}, \m{y}, \lambda^*\right), \quad \forall (\m{x}, \m{y}, \lambda) \in \mathcal{M}
\]
which can be also characterized as
\[
  \left \{\begin{array}{lll}
\textrm{For}\ i=1,2,\cdots,p,\\
\qquad x_{i}^* = \arg\min\limits_{x_{i}\in\mathcal{X}_{i}} \mathcal{L}_\beta\left(x_1^*,\cdots,x_{i-1}^*, x_i,x_{i+1}^*,\cdots, x_{p}^*,
\m{y}^*,\lambda^*\right),\\
\textrm{For}\ j=1,2,\cdots,q,\\
\qquad y_{j}^* = \arg\min\limits_{y_{j}\in\mathcal{Y}_{j}} \mathcal{L}_\beta\left(\m{x}^*,y_1^*,\cdots,y_{j-1}^*, y_j,y_{j+1}^*,\cdots, y_{q}^*,\lambda^*\right),\\
\lambda^* = \arg\max\limits_{ \lambda \in\mathcal{R}^n} \mathcal{L}_\beta\left(\m{x}^*,\m{y}^*,\lambda \right).
\end{array}\right.
\]
Then, by Lemma \ref{opt-1}, the above saddle-point equations can be equivalently expressed as
\begin{equation} \label{opt-3}
 \left \{\begin{array}{lll}
\textrm{For}\ i=1,2,\cdots,p,\\
\qquad x_i^*\in\mathcal{X}_i, \quad f_i(x_i)-f_i(x_i^*)+\langle x_i-x_i^*, -A_i \tr \lambda^*\rangle\geq 0,
 \forall x_i\in\mathcal{X}_i,\\
\textrm{For}\ j=1,2,\cdots,q,\\
\qquad
y_j^*\in\mathcal{Y}_j, \quad g_j(y_j)-g_j(y_j^*)+\langle y_j-y_j^*, -B_j \tr \lambda^*\rangle\geq 0,
 \forall y_j\in\mathcal{Y}_j,\\
\lambda^*\in\mathcal{R}^n, \quad \left\langle \lambda-\lambda^*, \C{A} \m{x}^* + \C{B} \m{y}^* - c \right\rangle\geq 0,
\forall\lambda\in\mathcal{R}^n.
\end{array}\right.
\end{equation}
Rewriting (\ref{opt-3}) in a more compact variational inequality (VI) form, we have
\begin{equation} \label{opt-vi-1}
 h(\m{u})- h(\m{u}^*) + \left\langle \m{w}-\m{w}^*, \mathcal{J}(\m{w}^*)\right\rangle \geq 0,\quad
\forall \m{w}\in \mathcal{M},
\end{equation}
where
\[
h(\m{u})=\sum_{i=1}^{p}f_i (x_i) +\sum_{j=1}^{p}g_j (y_j)
\]
and
\[
\m{u}=\left(\begin{array}{c}
\m{x}\\  \m{y}\\
\end{array}\right),\quad
\m{w}=\left(\begin{array}{c}
\m{x}\\ \m{y}\\  \lambda
\end{array}\right), \quad
\mathcal{J}(\m{w})=\left(\begin{array}{c}
-\C{A}\tr\lambda\\
-\C{B}\tr\lambda\\
 \C{A} \m{x} + \C{B} \m{y} - c
\end{array}\right).
\]
Noticing that the affine mapping $\C{J}$ is skew-symmetric, we immediately get
\begin{equation} \label{tilde-Jw}
\left\langle \m{w}-\widehat{\m{w}}, \mathcal{J}(\m{w})-\mathcal{J}(\widehat{\m{w}})\right\rangle= 0,
\quad \forall\ \m{w}, \widehat{\m{w}}\in \mathcal{M}.
\end{equation}
Hence, (\ref{opt-vi-1}) can be also rewritten as
\begin{equation} \label{opt-vi}
\textrm{VI}(\m{h}, \mathcal{J}, \mathcal{M}):\quad h(\m{u})- h(\m{u}^*) + \left\langle \m{w}-\m{w}^*, \mathcal{J}(\m{w})\right\rangle \geq 0,\quad
\forall \m{w} \in \mathcal{M}.
\end{equation}

Because of the nonempty assumption  on   the solution set of (\ref{Prob}),
the solution set $\mathcal{M}^*$ of the variational inequality  $\textrm{VI}(h, \mathcal{J}, \mathcal{M})$
is also nonempty and convex, see e.g. Theorem 2.3.5 \cite{FacchineiPang2003} for more details. The following theorem given by Theorem 2.1 \cite{HeYuan2012}
provides a concrete way to characterize the set $\mathcal{M}^*$.
\begin{theorem} \label{opt-thm}
The solution set of the variational inequality  $\textrm{VI}(h, \mathcal{J}, \mathcal{M})$ is convex and can be expressed as
\[
\mathcal{M}^*=\bigcap_{\m{w}\in \mathcal{M}}\left\{\widehat{\m{w}}\in \mathcal{M}|\ h(\m{u})-
h(\widehat{\m{u}})+ \left\langle \m{w}-\widehat{\m{w}}, \mathcal{J}(\m{w})\right\rangle\geq 0\right\}.
\]
\end{theorem}

\subsection{A prediction-correction interpretation of GS-ADMM}
Following a similar approach in \cite{HeMaYuan2016}, we next interpret
GS-ADMM  as a prediction-correction procedure. First,  let
\begin{equation} \label{tilde-xy}
\widetilde{\m{x}}^k=\left(\begin{array}{c}
\widetilde{x}_1^k\\  \widetilde{x}_2^k\\ \vdots\\ \widetilde{x}_p^k
\end{array}\right)=\left(\begin{array}{c}
x_1^{k+1}\\  x_2^{k+1}\\ \vdots\\ x_p^{k+1}
\end{array}\right),
\quad \widetilde{\m{y}}^k=\left(\begin{array}{c}
\widetilde{y}_1^k\\  \widetilde{y}_2^k\\ \vdots\\ \widetilde{y}_q^k
\end{array}\right)=\left(\begin{array}{c}
y_1^{k+1}\\  y_2^{k+1}\\ \vdots\\ y_q^{k+1}
\end{array}\right),
\end{equation}
\begin{equation} \label{tilde-lambda}
\widetilde{\lambda}^k=\lambda^k-\beta \left(\mathcal{A}\m{x}^{k+1}+\mathcal{B}\m{y}^{k}-c\right),
\end{equation}
and
\begin{equation}\label{tilde-uw}
\widetilde{\m{u}}=\left(\begin{array}{c}
\widetilde{\m{x}}\\  \widetilde{\m{y}}\\
\end{array}\right),\quad
\widetilde{\m{w}}^{k}=\left(\begin{array}{c}
 \widetilde{\m{x}}^k\\ \widetilde{\m{y}}^k\\ \widetilde{\lambda}^k
\end{array}\right)=
\left(\begin{array}{c}
\m{x}^{k+1}\\ \m{y}^{k+1}\\   \widetilde{\lambda}^{k}
\end{array}\right).
\end{equation}
Then, by using the above notations, we derive the  following basic lemma.
\begin{lemma}\label{tilde-VI}
For the iterates $\widetilde{\m{u}}^k, \widetilde{\m{w}}^k$ defined in (\ref{tilde-uw}),
we have $\widetilde{\m{w}}^k\in\mathcal{M}$ and
\begin{equation} \label{app-vi}
  h(\m{u})-h(\widetilde{\m{u}}^k)+ \left\langle \m{w}-\widetilde{\m{w}}^k, \mathcal{J}(\widetilde{\m{w}}^k)+Q(\widetilde{\m{w}}^k-\m{w}^k)\right\rangle \geq 0, \quad \forall \m{w}\in \mathcal{M},
\end{equation}
where
\begin{equation} \label{Q}
Q=\left[\begin{array}{cc}
H_{\mathbf{x}} & \bf{0}\\
\bf{0} & \widetilde{Q}
\end{array}\right]
\end{equation}
with
\begin{equation} \label{Hx}
H_{\mathbf{x}}=\beta\left[\begin{array}{ccccccccccccc}
\sigma_1 A_1\tr A_1 &&&& - A_1\tr A_2        &&&& \cdots &&&&- A_1\tr A_{p}  \\\\
- A_2\tr A_1        &&&& \sigma_1 A_2\tr A_2 &&&& \cdots &&&&- A_2\tr A_{p}  \\\\
\vdots              &&&& \vdots              &&&&\ddots  &&&&\vdots   \\\\
 - A_p\tr A_1       &&&&  - A_p\tr A_2       &&&&\cdots &&&&\sigma_1 A_p\tr A_p
\end{array}\right],
\end{equation}
\begin{equation} \label{tilde-Q}
 \widetilde{Q}=\left[\begin{array}{ccccccccccccccc|cc}
 (\sigma_2+1)\beta B_{1}\tr B_{1} &&&& {\bf 0} &&&&\cdots &&&&{\bf 0} &&&&-\tau B_{1}\tr \\
 &&&& &&&& &&&& &&&&\\
 {\bf 0} &&&& (\sigma_2+1)\beta B_{2}\tr B_{2} &&&&\cdots &&&&{\bf 0} &&&&-\tau B_{2}\tr \\
  &&&& &&&& &&&& &&&&\\
 \vdots &&&& \vdots &&&&\ddots &&&&\vdots &&&&\vdots \\
  &&&& &&&& &&&& &&&&\\
 {\bf 0} &&&& {\bf 0} &&&&\cdots &&&&(\sigma_2+1)\beta B_{q}\tr B_{q} &&&&-\tau B_{q}\tr \\
  &&&& &&&& &&&& &&&&\\ \hline
-B_{1} &&&& -B_{2}  &&&&\cdots &&&&-B_{q} &&&&\frac{1}{\beta}I
\end{array}\right].
\end{equation}
\end{lemma}
\begin{proof}
  Omitting  some constants, it is easy to verify that the $x_i$-subproblem $(i=1,2,\cdots,p)$ of GS-ADMM  can be written as
\[
x_i^{k+1}=\arg\min\limits_{x_i\in\mathcal{X}_i}\left\{
f_i(x_i)-\langle\lambda^k, A_ix_i \rangle
+\frac{\beta}{2}\left\|A_ix_i - c_{x,i}\right\|^2+\frac{\sigma_1\beta}{2}\left\|A_i(x_i-x_i^k)\right\|^2  \right\},
\]
where $c_{x,i} = c - \sum\limits_{l=1,l\neq i}^{p}A_lx_l^k - \mathcal{B}\m{y}^k$.
Hence, by Lemma \ref{opt-1}, we have $x_i^{k+1}\in\mathcal{X}_i$ and
\[
f_i(x_i)-f_i(x_i^{k+1})
+ \left\langle A_i (x_i-x_i^{k+1}),
 - \lambda^k + \beta (A_ix_i^{k+1}-c_{x,i})  +\sigma_1\beta A_i
 (x_i^{k+1}-x_i^k )  \right\rangle\geq 0
\]
for any $ x_i\in\mathcal{X}_i$. So, by the definition of (\ref{tilde-xy}) and
(\ref{tilde-lambda}), we get
\begin{equation}\label{opt-fx}
f_i(x_i)-f_i(\widetilde{x}_i^k)
+ \left\langle A_i (x_i-\widetilde{x}_i^k),
 - \widetilde{\lambda}^k  -\beta \sum_{l=1,l\neq i}^p A_l
 (\widetilde{x}_l^k-x_l^k )  +\sigma_1\beta \sum_{l=1}^p A_l
 (\widetilde{x}_l^k-x_l^k )  \right\rangle\geq 0
\end{equation}
for any $ x_i\in\mathcal{X}_i$.
By the way of generating $\lambda^{k+\frac{1}{2}}$ in (\ref{GS-ADMM}) and
the definition of (\ref{tilde-lambda}),  the following relation holds
\begin{equation}\label{eq-lambda}
\lambda^{k+\frac{1}{2}}=\lambda^k-\tau (\lambda^k- \widetilde{\lambda}^k\ ).
\end{equation}
Similarly, the $y_j$-subproblem ($j=1,\cdots,q$) of GS-ADMM can be written
as
\[
y_j^{k+1}=\arg\min\limits_{y_j\in\mathcal{Y}_j}\left\{
g_j(y_j)-\left\langle\lambda^{k+\frac{1}{2}}, B_jy_j\right\rangle
+\frac{\beta}{2}\left\|B_jy_j - c_{y,j}\right\|^2+\frac{\sigma_2\beta}{2}\left\|B_j(y_j-y_j^k)\right\|^2
\right\},
\]
where $c_{y,j} = c - \mathcal{A}\m{x}^{k+1} - \sum\limits_{l=1,l\neq j}^{q}B_ly_l^k$.
Hence,  by Lemma \ref{opt-1}, we have $y_j^{k+1}\in\mathcal{Y}_j$ and
\[
g_j(y_j)-g_j(y_j^{k+1})+ \left\langle B_j (y_j-y_j^{k+1}),
 - \lambda^{k+\frac{1}{2}} + \beta (B_jy_j^{k+1} - c_{y,j})
 +\sigma_2\beta  B_j(y_j^{k+1}-y_j^k)
\right\rangle\geq 0
 \]
for any $y_j \in \C{Y}_j$. This inequality, by using (\ref{eq-lambda}) and
the definition of (\ref{tilde-xy}) and (\ref{tilde-lambda}), can be rewritten as
\begin{equation}\label{opt-gy}
 g_j(y_j)-g_j(\widetilde{y}_j^{k}) +
  \left\langle B_j (y_j-\widetilde{y}_j^{k}), - \widetilde{\lambda}^k+(\sigma_2+1) \beta B_j(\widetilde{y}_j^{k}-y_j^k)-\tau (\widetilde{\lambda}^k-\lambda^k)\right\rangle\geq 0
\end{equation}
for any $y_j \in \C{Y}_j$.
Besides, the equality (\ref{tilde-lambda}) can be rewritten as
\[
\left(\mathcal{A}\widetilde{\m{x}}^{k}+\mathcal{B}\widetilde{\m{y}}^{k}-c\right)-\mathcal{B}\left(\widetilde{\m{y}}^{k}-\m{y}^k\right)
+\frac{1}{\beta} (\widetilde{\lambda}^k-\lambda^k)=0,
\]
which is   equivalent to
\begin{equation}\label{opt-wlamb}
\left\langle \lambda-\widetilde{\lambda}^k, (\mathcal{A}\widetilde{\m{x}}^{k}+\mathcal{B}\widetilde{\m{y}}^{k}-c)-\mathcal{B} (\widetilde{\m{y}}^{k}-\m{y}^k)
+\frac{1}{\beta} (\widetilde{\lambda}^k-\lambda^k)\right\rangle\geq0,\ \forall \lambda\in \mathcal{R}^n.
\end{equation}
Then, (\ref{app-vi}) follows from (\ref{opt-fx}), (\ref{opt-gy}) and (\ref{opt-wlamb}). $\ \ \ \diamondsuit$
\end{proof}

\begin{lemma} \label{pre-lem}
 For the sequences  $\{\m{w}^k\}$ and $\{\widetilde{\m{w}}^k\}$ generated by  GS-ADMM, the following equality  holds
\begin{equation} \label{pre-core}
\m{w}^{k+1}=\m{w}^k-M(\m{w}^k-\widetilde{\m{w}}^k),
\end{equation}
 where
\begin{equation}\label{M}
M=\left[\begin{array}{c|cccc}
I &     &  & & \\ \hline
&   I &  & & \\
&    &\ddots  & & \\
&    &  &I & \\
&  -s\beta B_1 &\cdots  & -s\beta B_q & (\tau+s)I \\
\end{array}\right].
 \end{equation}
\end{lemma}
\begin{proof} It follows from the way of generating $\lambda^{k+1}$ in the algorithm and (\ref{eq-lambda}) that
\[
\begin{array}{lll}
\lambda^{k+1}  &=& \lambda^{k+\frac{1}{2}}-s\beta(\mathcal{A}\m{x}^{k+1}+\mathcal{B}\m{y}^{k+1}-c)    \\
 &=&  \lambda^{k+\frac{1}{2}}-s\beta(\mathcal{A}\m{x}^{k+1}+\mathcal{B}\m{y}^{k}-c)+s\beta \mathcal{B}(\m{y}^{k}-\m{y}^{k+1})  \\
 &=& \lambda^k-\tau(\lambda^k- \widetilde{\lambda}^k)-s(\lambda^k- \widetilde{\lambda}^k)+s\beta \mathcal{B}(\m{y}^{k}-\widetilde{\m{y}}^{k}) \\
&=& \lambda^k-\left[-s\beta \mathcal{B}(\m{y}^{k}-\widetilde{\m{y}}^{k})+(\tau+s)(\lambda^k- \widetilde{\lambda}^k) \right].
\end{array}
\]
The above equality together with $x_i^{k+1}=\widetilde{x}_l^k$, for $i=1,\cdots,p$,
and $y_j^{k+1}=\widetilde{y}_j^k$, for $j=1,\cdots,q$, imply
\[
{\footnotesize\left(\begin{array}{c}
x_1^{k+1}\\  \vdots\\ x_p^{k+1}\\ y_1^{k+1}\\  \vdots\\ y_q^{k+1}\\ \lambda^{k+1}
\end{array}\right)
=\left(\begin{array}{c}
x_1^{k}\\  \vdots\\ x_p^{k}\\ y_1^{k}\\  \vdots\\ y_q^{k}\\ \lambda^{k}
\end{array}\right)-\left[\begin{array}{ccc|cccc}
I  &   &   &  &   &   &\\
  &    \ddots & &   &   &  & \\
  &   &   I &  &   &  & \\ \hline
  &  &   &  I &  & & \\
    &  &   &   &\ddots  & & \\
        &  &   &   &  &I & \\
  &  & & -s\beta B_1 &\cdots  & -s\beta B_q & (\tau+s)I \\
\end{array}\right]\left(\begin{array}{c}
x_1^k-\widetilde{x}_1^{k}\\ \vdots\\  x_p^k-\widetilde{x}_p^{k}\\ y_1^k-\widetilde{y}_1^{k}\\ \vdots\\  y_q^k-\widetilde{y}_q^{k}\\ \lambda^{k}- \widetilde{\lambda}^{k}
\end{array}\right), }
\]
which immediately gives (\ref{pre-core}). $\ \ \ \diamondsuit$
\end{proof}

Lemma \ref{tilde-VI} and Lemma \ref{pre-lem} show that our GS-ADMM can be interpreted as a prediction-correction framework,
where  $\m{w}^{k+1}$ and $\widetilde{\m{w}}^k$ are normally  called the predictive
variable  and the correcting variable, respectively.

\section{Convergence analysis of GS-ADMM}
Compared with (\ref{opt-vi}) and (\ref{app-vi}),
the key to  proving  the convergence of GS-ADMM is to verify that the   extra term in (\ref{app-vi}) converges to zero, that is,
\[
\lim_{k\rightarrow \infty} \left\langle \m{w}-\widetilde{\m{w}}^k,Q  (\m{w}^k-\widetilde{\m{w}}^k) \right\rangle=0,\quad \forall \m{w}\in \mathcal{M}.
 \]
In this section, we first investigate some properties of the sequence
$\{\|\m{w}^k-\m{w}^*\|_H^2\}$. Then, we provide a lower bound of
$\|\m{w}^k-\widetilde{\m{w}}^k\|_G^2$. Based on these properties,
the global convergence and  worst-case $\mathcal{O}(1/t)$
convergence rate of GS-ADMM are established in the end.

\subsection{Properties of $\{\left\|\m{w}^k-\m{w}^*\right\|_H^2\}$}
It follows from (\ref{tilde-Jw}) and (\ref{app-vi}) that
$ \widetilde{\m{w}}^k\in\mathcal{M}$ and
\begin{equation} \label{VI31}
 \quad   h(\m{u})-h(\widetilde{\m{u}}^k)+ \left\langle \m{w}-\widetilde{\m{w}}^k,
\mathcal{J}(\m{w})\right\rangle \geq \left\langle \m{w}-\widetilde{\m{w}}^k, Q(\m{w}^k-\widetilde{\m{w}}^k)\right\rangle, \quad
\forall \m{w}\in \mathcal{M}.
\end{equation}
Suppose $\tau + s >0$. Then, the matrix $M$ defined in (\ref{M}) is nonsingular.
So, by (\ref{pre-core}) and a direct calculation,
the right-hand term of (\ref{VI31}) is  rewritten as
\begin{equation}\label{eq-VI}
\left\langle \m{w}-\widetilde{\m{w}}^k, Q(\m{w}^k-\widetilde{\m{w}}^k)\right\rangle=\left\langle \m{w}-\widetilde{\m{w}}^k,
H(\m{w}^k-\m{w}^{k+1})\right\rangle
\end{equation} where
\begin{equation} \label{H}
H=QM^{-1}=\left[\begin{array}{c|c}
H_{\mathbf{x}} & \bf{0} \\ \hline
\bf{0 } & \widetilde{H}
\end{array}\right] \end{equation}
with  $H_x$ defined in (\ref{Hx}) and
\begin{equation}\label{tilde-H}
\widetilde{H}=\left[\begin{array}{ccccccccccc|cc}
(\sigma_2+1-\frac{\tau s}{\tau+s})\beta B_1\tr B_1  &&&&\cdots &&&&-\frac{\tau s}{\tau+s} \beta B_1\tr B_q &&&&-\frac{\tau }{\tau+s} B_1\tr \\
 &&&&   &&&&  &&&&\\
   \vdots  &&&&\ddots &&&&\vdots &&&&\vdots \\
    &&&&   &&&&  &&&&\\
-\frac{\tau s}{\tau+s}\beta B_q\tr B_1  &&&&\cdots &&&&(\sigma_2+1-\frac{\tau s}{\tau+s})\beta B_q\tr B_q &&&&-\frac{\tau }{\tau+s} B_q\tr  \\
 &&&&   &&&&  &&&&\\ \hline
     -\frac{\tau }{\tau+s}B_1 &&&&\cdots &&&&-\frac{\tau }{\tau+s}B_q &&&&\frac{1 }{\beta(\tau+s)}I
\end{array}\right].
\end{equation}

The following lemma shows that $H$ is a positive definite matrix for proper choice of the parameters  $(\sigma_1,\sigma_2, \tau, s)$.
\begin{lemma} \label{spd-H}
 The matrix $H$ defined in (\ref{H}) is symmetric positive definite if
\begin{equation} \label{parameter-region}
\sigma_1\in (p-1,+\infty),\quad \sigma_2\in (q-1,+\infty),\quad \tau \le 1 \quad \mbox{and} \quad \tau+s >0.
 \end{equation}
\end{lemma}
\begin{proof}
By the block structure of $H$, we only need to show that the blocks $H_\mathbf{x}$ and  $\widetilde{H}$ in (\ref{H}) are positive definite
if the parameters $(\sigma_1,\sigma_2, \tau, s)$ satisfy (\ref{parameter-region}). Note that
\begin{equation}\label{Hx-decomp}
H_\mathbf{x}=\beta\left[\begin{array}{cccc}
A_1 &  & & \\
  & A_2&  &  \\
  &  &\ddots & \\
  &  &  &A_p
\end{array}\right]\tr H_{\mathbf{x},0} \left[\begin{array}{cccc}
A_1 &  & & \\
  & A_2&  &  \\
  &  &\ddots & \\
  &  &  &A_p
\end{array}\right],
\end{equation}
where
\begin{equation}\label{Hx0}
 H_{\mathbf{x},0}=\left[\begin{array}{cccc}
\sigma_1I &-I & \cdots&-I\\
-I & \sigma_1I&\cdots &-I \\
\vdots &\vdots &\ddots &\vdots\\
 -I& -I&\cdots &\sigma_1I
\end{array}\right]_{p \times p}.
\end{equation}
If $\sigma_1\in (p-1,+\infty)$, $H_{\mathbf{x},0}$ is positive definite.
Then, it follows from (\ref{Hx-decomp}) that $H_\mathbf{x}$ is positive definite
if $\sigma_1\in (p-1,+\infty)$ and  all $A_i$, $i=1, \cdots, p$,  have full column rank.

Now, note that the matrix $\widetilde{H}$ can be decomposed as
\begin{equation}\label{tildeH-decomp}
\widetilde{H}=\widetilde{D}\tr \widetilde{H}_0\widetilde{D},
\end{equation}
where
\begin{equation}\label{titl-Dw}
\widetilde{D}= \left[\begin{array}{cccc|c}
\beta^{\frac{1}{2}}B_1 &   &   &  &  \\
  &\beta^{\frac{1}{2}}B_2  &  &  &   \\
  &   & \ddots &   &    \\
  &   &   & \beta^{\frac{1}{2}}B_q &  \\ \hline
  &  &   & & \beta^{-\frac{1}{2}}I  \\
\end{array}\right]
\end{equation}
and
\[
\widetilde{H}_0=\left[\begin{array}{cccccccccccccc|cc}
\left(\sigma_2+1-\frac{\tau s}{\tau+s}\right) I  &&&& -\frac{\tau s}{\tau+s} I  &&&&\cdots &&&&-\frac{\tau s}{\tau+s} I  &&& -\frac{\tau }{\tau+s}I   \\
  &&&&   &&&& &&&&  &&&  \\
-\frac{\tau s}{\tau+s} I  &&&& \left(\sigma_2+1-\frac{\tau s}{\tau+s}\right) I  &&&&\cdots &&&&-\frac{\tau s}{\tau+s} I  &&& -\frac{\tau }{\tau+s}I   \\
 &&&&   &&&& &&&&  &&&  \\
\vdots &&&& \vdots &&&&\ddots &&&&\vdots &&& \vdots  \\
 &&&&   &&&& &&&&  &&&  \\
 -\frac{\tau s}{\tau+s} I &&&& -\frac{\tau s}{\tau+s} I  &&&&\cdots &&&&\left(\sigma_2+1-\frac{\tau s}{\tau+s}\right) I  &&& -\frac{\tau }{\tau+s}I   \\
  &&&&   &&&& &&&&  &&&  \\
 \hline
 -\frac{\tau }{\tau+s} I &&&& -\frac{\tau }{\tau+s} I  &&&&\cdots &&&&-\frac{\tau }{\tau+s} I  &&& \frac{1 }{\tau+s}I   \\
\end{array}\right].
\]
According to the fact that
 \begin{eqnarray*}
&&\left[\begin{array}{ccccc}
I  & & && \tau I \\
 & &\ddots && \tau I \\
 & & &&  I
\end{array}\right] \ \widetilde{H}_0 \ \left[\begin{array}{ccccc}
I  & & && \tau I \\
 & &\ddots && \tau I \\
 & & &&  I
\end{array}\right]\tr \\
& =&
\left[\begin{array}{cccccccccccccc|cc}
(\sigma_2+1-\tau)I  &&&& -\tau I &&&& \cdots&&&&-\tau I &&& {\bf 0}\\
   &&&&   &&&&  &&&&  &&&  \\
-\tau I  &&&& (\sigma_2+1-\tau)I &&&& \cdots&&&& -\tau I &&&{\bf 0}\\
 &&&&   &&&&  &&&&  &&&  \\
 \vdots&&&&\vdots &&&&\ddots &&&&\vdots&&& \vdots\\
  &&&&   &&&&  &&&&  &&&  \\
-\tau I  &&&&-\tau I  &&&&\cdots&&&& (\sigma_2+1-\tau)I &&& {\bf 0}\\
 &&&&   &&&&  &&&&  &&&  \\\hline
 {\bf 0}&&&&{\bf 0} &&&&\cdots &&&&{\bf 0}& &&\frac{1}{\tau+s} I
\end{array}\right]\\
&=& \left[\begin{array}{ccccc}
H_{\mathbf{y},0}+(1-\tau)EE\tr &&&& \bf{0} \\
  &&&&   \\
\bf{0}  &&&& \frac{1}{\tau+s} I
\end{array}\right],
\end{eqnarray*}
where
\begin{equation} \label{titl-EHy0}
E=\left[\begin{array}{c}
I  \\
I \\
\vdots\\
I
\end{array}\right] \quad \mbox{ and } \quad
H_{\mathbf{y},0}=\left[\begin{array}{cccc}
\sigma_2 I  & -I & \cdots &-I \\
-I  & \sigma_2 I  &\cdots &-I \\
 \vdots& \vdots& \ddots& \vdots\\
-I  & -I  &\cdots&  \sigma_2 I
\end{array}\right]_{q \times q},
\end{equation}
we have $\widetilde{H}_0$ is positive definite if and only if $H_{\mathbf{y},0}+(1-\tau)EE\tr$
is positive definite and $\tau+ s >0$. Note that $H_{\mathbf{y},0}$ is positive definite
if $\sigma_2\in (q-1,+\infty)$, and $(1-\tau)EE\tr$ is positive semidefinite if $\tau \le 1$.
So,  $\widetilde{H}_0$ is positive definite if
$\sigma_2\in (q-1,+\infty)$, $ \tau \le 1$ and $\tau + s >0$.
Then, it follows from (\ref{tildeH-decomp}) that $\widetilde{H}$ is positive definite,  if
$\sigma_2\in (q-1,+\infty)$,  $ \tau \le 1$, $\tau + s >0$  and all the matrices $B_j$, $j=1, \cdots, q$,  have full column rank.

Summarizing the above discussions, the matrix $H$ is positive definite if  the parameters $(\sigma_1,\sigma_2, \tau, s)$ satisfy (\ref{parameter-region}). $\ \ \ \diamondsuit$
\end{proof}

\begin{theorem} \label{ther-32}
The sequences  $\{\m{w}^k\}$ and $\{\widetilde{\m{w}}^k\}$ generated by GS-ADMM satisfy
\begin{eqnarray}\label{Ieq-32}
  h(\m{u})-h(\widetilde{\m{u}}^k)+ \left\langle \m{w}-\widetilde{\m{w}}^k, \mathcal{J}(\m{w}) \right\rangle
& \ge & \frac{1}{2}\left\{ \left\|\m{w}-\m{w}^{k+1}\right\|_H^2-\left\|\m{w}-\m{w}^{k}\right\|_H^2\right\}\nonumber\\
&&+\frac{1}{2}\left\|\m{w}^{k}-\widetilde{\m{w}}^k\right\|_G^2, \
\forall \m{w}\in \mathcal{M},
\end{eqnarray}
where
\begin{equation}\label{eq-G}
G=Q+Q\tr -M\tr HM. \end{equation}
\end{theorem}
\begin{proof}
By  substituting
\[
a=\m{w},\ b=\widetilde{\m{w}}^k,\ c=\m{w}^k,\ d=\m{w}^{k+1},
\]
into the following identity
\[
2\langle a-b, H(c-d)\rangle=\|a-d\|_H^2- \|a-c\|_H^2+\|c-b\|_H^2- \|d-b\|_H^2,
\]
we have
\begin{eqnarray}\label{ident}
&& 2 \left\langle \m{w}-\widetilde{\m{w}}^k, H(\m{w}^k-\m{w}^{k+1}) \right\rangle \nonumber\\
&=& \left\|\m{w}-\m{w}^{k+1}\right\|_H^2- \left\|\m{w}-\m{w}^k\right\|_H^2
+\left\|\m{w}^{k}-\widetilde{\m{w}}^k\right\|_H^2- \left\|\m{w}^{k+1}-\widetilde{\m{w}}^k\right\|_H^2.
\end{eqnarray}
Now, notice that
\begin{eqnarray} \label{ident-H}
&& \left\|\m{w}^{k}-\widetilde{\m{w}}^k\right\|_H^2 - \left\|\m{w}^{k+1}-\widetilde{\m{w}}^k\right\|_H^2 \nonumber \\
& =&  \left\|\m{w}^{k}-\widetilde{\m{w}}^k\right\|_H^2 - \left\|\widetilde{\m{w}}^k-\m{w}^k+\m{w}^k-\m{w}^{k+1}\right\|_H^2 \nonumber \\
& =& \left\|\m{w}^{k}-\widetilde{\m{w}}^k\right\|_H^2- \left\|\widetilde{\m{w}}^k-\m{w}^k+M(\m{w}^k-\widetilde{\m{w}}^{k})\right\|_H^2  \nonumber \\
& =& \left\langle \m{w}^k-\widetilde{\m{w}}^k, \left(HM+(HM)\tr -M\tr HM\right)(\m{w}^k-\widetilde{\m{w}}^{k}) \right\rangle \nonumber \\
& = &\left\langle \m{w}^k-\widetilde{\m{w}}^k, \left(Q+Q\tr -M\tr HM\right)(\m{w}^k-\widetilde{\m{w}}^{k}) \right\rangle,
\end{eqnarray}
where the second equality holds by (\ref{pre-core}) and the fourth equality follows from (\ref{H}).
Then, (\ref{Ieq-32}) follows from (\ref{VI31})-(\ref{eq-VI}), (\ref{ident})-(\ref{ident-H}) and the definition of $G$ in (\ref{eq-G}). $\ \ \ \diamondsuit$
\end{proof}

\begin{theorem} \label{Ieq-33}
The sequences  $\{\m{w}^k\}$ and $\{\widetilde{\m{w}}^k\}$ generated by GS-ADMM satisfy
\begin{equation} \label{Contra-wH}
\left\|\m{w}^{k+1}-\m{w}^*\right\|_H^2\leq \left\|\m{w}^{k}-\m{w}^*\right\|_H^2-\left\|\m{w}^{k}-\widetilde{\m{w}}^k\right\|_G^2,\ \forall \m{w}^*\in \mathcal{M}^*.
\end{equation}
\end{theorem}
\begin{proof} Setting $\mathbf{w}=\m{w}^* \in \mathcal{M}^*$ in (\ref{Ieq-32}) we have
\[
 \frac{1}{2} \left\|\m{w}^{k+1}-\m{w}^*\right\|_H^2 \le \frac{1}{2} \left\|\m{w}^{k}-\m{w}^*\right\|_H^2 - \frac{1}{2} \left\|\m{w}^{k}-\widetilde{\m{w}}^k\right\|_G^2
+ h(\m{u}^*)-h(\widetilde{\m{u}}^k)+ \left\langle \m{w}^*-\widetilde{\m{w}}^k, \mathcal{J}(\m{w}^*) \right\rangle. \nonumber
\]
The above inequality together with  (\ref{opt-vi-1})  reduces to the inequality (\ref{Contra-wH}). $\ \ \ \diamondsuit$
\end{proof}

It can be observed that if $\left\|\m{w}^k-\widetilde{\m{w}}^k\right\|_G^2$ is positive, then
the inequality (\ref{Contra-wH}) implies the contractiveness of
the sequence $\{\m{w}^k-\m{w}^*\}$ under the $H$-weighted norm.
However, the matrix $G$ defined in (\ref{eq-G}) is not necessarily positive definite
 when $\sigma_1\in (p-1,+\infty)$, $\sigma_2\in (q-1,+\infty)$ and $(\tau, s) \in \C{K}$.
Therefore, it is necessary to estimate the lower bound of $\left\|\m{w}^k-\widetilde{\m{w}}^k\right\|_G^2$ for the sake of the   convergence analysis.

\subsection{Lower bound of $\left\|\m{w}^k-\widetilde{\m{w}}^k\right\|_G^2$ }
This subsection provides a lower bound of  $\left\|\m{w}^k-\widetilde{\m{w}}^k\right\|_G^2$, for
$\sigma_1\in (p-1,+\infty)$, $\sigma_2\in (q-1,+\infty)$ and $(\tau, s) \in \C{K}$, where
  $\C{K}$ is  defined in (\ref{setK}).

By   simple calculations, the $G$ given in (\ref{eq-G}) can be explicitly written as
\[
G=\left[\begin{array}{c|c}
H_{\mathbf{x}} & \bf{0} \\ \hline
\bf{0}  & \widetilde{G}
\end{array}\right],
\]
 where $H_\mathbf{x}$ is defined in (\ref{Hx}) and
\begin{equation}\label{tilde-G}
\widetilde{G}=\left[\begin{array}{ccccccccccccc|cc}
(\sigma_2+1-s)\beta B_1\tr B_1 &&&&\cdots &&&&-s\beta B_1\tr B_q &&&&&&(s-1) B_1\tr \\
      &&&&  &&&&  &&&&&&  \\
   \vdots  &&&&\ddots &&&&\vdots &&&&&&\vdots \\
   &&&&  &&&&  &&&&&&  \\
-s\beta B_q\tr B_1   &&&&\cdots &&&&(\sigma_2+1-s)\beta B_q\tr B_q &&&&&&s-1)  B_q\tr \\
&&&&  &&&&  &&&&&&  \\\hline
     (s-1) B_1 &&&&\cdots &&&&(s-1) B_q &&&&&&\frac{2-\tau-s }{\beta}I
\end{array}\right].
\end{equation}
In addition, we have
\[
\widetilde{G}=\widetilde{D}\tr \widetilde{G}_0\widetilde{D},
\]
where $\widetilde{D}$ is defined in (\ref{titl-Dw}) and
\begin{equation} \label{tilde-G0}
\widetilde{G}_0=\left[\begin{array}{ccccccccccccccccc|cc}
(\sigma_2+1-s) I  &&&& -s I  &&&&\cdots &&&&-s I  &&&&&& (s-1)I   \\
  &&&&   &&&&  &&&&  &&&&&&   \\
-s I  &&&& (\sigma_2+1-s) I  &&&&\cdots &&&&-s I  &&&&&&  (s-1)I   \\
  &&&&   &&&&  &&&&  &&&&&&   \\
\vdots &&&& \vdots &&&&\ddots &&&&\vdots &&&&&& \vdots  \\
  &&&&   &&&&  &&&&  &&&&&&   \\
 -s I &&&& -s I  &&&&\cdots &&&&(\sigma_2+1-s) I  &&&&&& (s-1)I   \\
   &&&&   &&&&  &&&&  &&&&&&   \\
 \hline
 (s-1) I &&&& (s-1) I  &&&&\cdots &&&&(s-1) I  &&&&&& (2-\tau-s)I   \\
\end{array}\right].
\end{equation}

Now,  we present the following lemma which provides a lower bound of $\|\m{w}^k-\widetilde{\m{w}}^k\|_G$.
\begin{lemma}
Suppose $\sigma_1\in (p-1,+\infty)$ and $\sigma_2\in (q-1,+\infty)$.
For the sequences $\{\m{w}^k\}$ and $\{ \widetilde{\m{w}}^k \}$
generated by GS-ADMM, there exists $\xi_1>0$ such that
\begin{eqnarray} \label{lower-wG1}
 \left\|\m{w}^k-\widetilde{\m{w}}^k\right\|_G^2 &\geq& \xi_1 \left(\sum\limits_{i=1}^{p}\left\|A_i\left(x_i^k-x_i^{k+1}\right)\right\|^2
 +\sum\limits_{j=1}^{q}\left\|B_j\left(y_j^k-y_j^{k+1}\right)\right\|^2\right) \nonumber \\
&& +  (1-\tau)\beta\left\|\mathcal{B}\left(\m{y}^{k}-\m{y}^{k+1}\right)\right\|^2+ (2-\tau-s)\beta \left\|\mathcal{A}\m{x}^{k+1}+\mathcal{B}\m{y}^{k+1}-c\right\|^2 \nonumber \\
& &+   2(1-\tau)\beta\left(\mathcal{A}\m{x}^{k+1}+\mathcal{B}\m{y}^{k+1}-c\right)\tr \mathcal{B}\left(\m{y}^{k}-\m{y}^{k+1}\right).
\end{eqnarray}
\end{lemma}
\begin{proof} First, it is easy to derive  that
\[
\|\m{w}^k-\widetilde{\m{w}}^k\|_G^2=\left\|\left(\begin{array}{l}
A_1\left(x_1^k-x_1^{k+1}\right)\\
\ \ \ \ \ \vdots\\
A_{p}\left(x_p^k-x_p^{k+1}\right)\\
B_1\left(y_1^k-y_1^{k+1}\right)\\
\ \ \ \ \ \vdots\\
B_{q}\left(y_q^k-y_q^{k+1}\right)\\
\mathcal{A}\m{x}^{k+1}+\mathcal{B}\m{y}^{k+1}-c
\end{array}\right)
\right\|_{ \overline{G}}^2,
\]
where
\[
\overline{G}=
\beta\left[\begin{array}{l|llll}
I & & & &  \\ \hline
&I &&&I \\
&&&\ddots&\vdots\\
&&& & I
\end{array}\right]
\left[\begin{array}{lllll}
H_{\mathbf{x},0} &&&&  \\
&&&& \widetilde{G}_0
\end{array}\right]\left[\begin{array}{l|llll}
I  & & & &  \\ \hline
&I &&&I \\
&&&\ddots&\vdots\\
&&& & I
\end{array}\right]\tr
=\beta \left[\begin{array}{lllll}
H_{\mathbf{x},0}  &&&&\\
&&&&   \overline{G}_0
\end{array}\right],
\]
with $H_{\mathbf{x},0}$ and $\widetilde{G}_0$  defined  in (\ref{Hx0}) and (\ref{tilde-G0}), respectively,
and
\[
\begin{array}{lll}
\overline{G}_0&=&
\left[\begin{array}{lllllllllllllllll|ll}
(\sigma_2+1-\tau)I  &&&& -\tau I&&&&  \cdots &&&& -\tau I &&&&&& (1-\tau) I \\
 &&&&  &&&&   &&&&  &&&&&&  \\
-\tau I&&&& (\sigma_2+1-\tau)I &&&& \cdots&&&& -\tau I &&&&&&(1-\tau) I \\
 &&&&  &&&&   &&&&  &&&&&&  \\
\ \ \vdots&&&& \ \ \vdots&&&& \ddots &&&& \ \ \vdots&&&&&&\ \ \vdots \\
 &&&&  &&&&   &&&&  &&&&&&  \\
-\tau I&&&& -\tau I&&&& \cdots&&&& (\sigma_2+1-\tau)I&&&&&&(1-\tau) I\\
 &&&&  &&&&   &&&&  &&&&&&  \\\hline
(1-\tau) I&&&& (1-\tau) I&&&& \cdots&&&&  (1-\tau) I&&&&&& (2-\tau-s) I
\end{array}\right]\\\\
&=& \left[\begin{array}{llll|ll}
&&&&& (1-\tau) I\\
&H_{\mathbf{y},0}+(1-\tau)EE\tr &&&&\ \ \vdots\\
&&&&& (1-\tau) I\\ \hline
(1-\tau) I&\ \ \ \ \ \  \cdots&(1-\tau) I&&& (2-\tau-s) I
\end{array}\right].
\end{array}
\]
In the above matrix $\overline{G}_0$, $E$ and $H_{\mathbf{y},0}$ are defined in (\ref{titl-EHy0}).
Hence, we have
\begin{eqnarray}\label{eq-wGnorm}
\frac{1}{\beta} \left\|\m{w}^k-\widetilde{\m{w}}^k\right\|_G^2
& =&   \left\|\left(\begin{array}{l}
A_1\left(x_1^k-x_1^{k+1}\right)\\
A_2\left(x_2^k-x_2^{k+1}\right)\\
\ \ \ \ \ \vdots\\
A_{p}\left(x_{p}^k-x_{p}^{k+1}\right)
\end{array}\right)
\right\|_{ H_{\mathbf{x},0}}^2
+ \left\|\left(\begin{array}{l}
B_1\left(y_1^k-y_1^{k+1}\right)\\
B_2\left(y_2^k-y_2^{k+1}\right)\\
\ \ \ \ \ \vdots\\
B_{q}\left(y_{q}^k-y_{q}^{k+1}\right)
\end{array}\right)
\right\|_{H_{\mathbf{y},0}}^2 \nonumber \\
&&  + \left\|\left(\begin{array}{l}
B_1\left(y_1^k-y_1^{k+1}\right)\\
B_2\left(y_2^k-y_2^{k+1}\right)\\
\ \ \ \ \ \vdots\\
B_{q}\left(y_{q}^k-y_{q}^{k+1}\right)
\end{array}\right)
\right\|_{(1-\tau)EE\tr}^2+(2-\tau-s) \left\|\mathcal{A}\m{x}^{k+1}+\mathcal{B}\m{y}^{k+1}-c\right\|^2 \nonumber \\
&& + 2(1-\tau) \left(\mathcal{A}\m{x}^{k+1}+\mathcal{B}\m{y}^{k+1}-c\right)\tr \mathcal{B}\left(\m{y}^{k}-\m{y}^{k+1}\right).
\end{eqnarray}
Since  $\sigma_1\in (p-1,+\infty)$ and $\sigma_2\in (q-1,+\infty)$, $H_{\mathbf{x},0}$ and $H_{\mathbf{y},0}$ are
positive definite. So, there exists a $\xi_1 >0$ such that
\begin{eqnarray}\label{first-term}
& &  \left\|\left(\begin{array}{l}
A_1\left(x_1^k-x_1^{k+1}\right)\\
A_2\left(x_2^k-x_2^{k+1}\right)\\
\ \ \ \ \ \vdots\\
A_{p}\left(x_{p}^k-x_{p}^{k+1}\right)
\end{array}\right)
\right\|_{ H_{\mathbf{x},0}}^2
+ \left\|\left(\begin{array}{l}
B_1\left(y_1^k-y_1^{k+1}\right)\\
B_2\left(y_2^k-y_2^{k+1}\right)\\
\ \ \ \ \ \vdots\\
B_{q}\left(y_{q}^k-y_{q}^{k+1}\right)
\end{array}\right)
\right\|_{H_{\mathbf{y},0}}^2 \\
& \ge & \xi_1 \left(\sum\limits_{i=1}^{p}\|A_i\left(x_i^k-x_i^{k+1}\right)\|^2+\sum\limits_{j=1}^{q}\|B_j\left(y_j^k-y_j^{k+1}\right)\|^2\right). \nonumber
\end{eqnarray}
In view of the definition of $E$ in (\ref{titl-EHy0}), we have
\[
\left\|\left(\begin{array}{l}
B_1\left(y_1^k-y_1^{k+1}\right)\\
B_2\left(y_2^k-y_2^{k+1}\right)\\
\ \ \ \ \ \vdots\\
B_{q}\left(y_{q}^k-y_{q}^{k+1}\right)
\end{array}\right)
\right\|_{(1-\tau)EE\tr}^2=(1-\tau)\left\|\mathcal{B}\left(\m{y}^{k}-\m{y}^{k+1}\right)\right\|^2.
\]
Then, the inequality (\ref{lower-wG1}) follows from (\ref{eq-wGnorm}) and (\ref{first-term}). $\ \ \ \diamondsuit$
\end{proof}

\begin{lemma} \label{Lem-36}
Suppose $ \tau>-1$. Then, the sequence $\{\m{w}^k\}$
generated by   GS-ADMM satisfies
\begin{eqnarray} \label{Ieq-ABw}
&& \left(\mathcal{A}\m{x}^{k+1}+\mathcal{B}\m{y}^{k+1}-c\right)\tr \mathcal{B}\left(\m{y}^{k}-\m{y}^{k+1}\right) \\
 &\ge & \frac{1-s}{1+\tau}\left(\mathcal{A}\m{x}^{k}+\mathcal{B}\m{y}^{k}-c\right)\tr \mathcal{B}\left(\m{y}^{k}-\m{y}^{k+1}\right)
 -\frac{\tau}{1+\tau}\left\| \mathcal{B}\left(\m{y}^{k}-\m{y}^{k+1}\right)\right\|^2 \nonumber \\
 &+& \frac{1}{2(1+\tau)\beta}\left(\left\|\m{y}^{k+1}-\m{y}^k\right\|^2_{H_\mathbf{y}}-\left\|\m{y}^{k}-\m{y}^{k-1}\right\|^2_{H_\mathbf{y}}\right). \nonumber
\end{eqnarray}
\end{lemma}
\begin{proof} It follows from the optimality condition of $y_l^{k+1}$-subproblem that
$y_l^{k+1}\in\mathcal{Y}_l$ and for any $y_l \in \C{Y}_l$, we have
\[
 g_l(y_l)-g_l(y_l^{k+1}) + \left\langle B_l (y_l-y_l^{k+1}),  - \lambda^{k+\frac{1}{2}}+\sigma_2\beta  B_l\left(y_l^{k+1}-y_l^k\right)
+ \beta (B_l y_l^{k+1} -c_{y,l}) \right\rangle \ge 0
\]
with $c_{y,l} = c - \C{A}\m{x}^{k+1} - \sum\limits_{j=1,j\neq l}^{q}B_jy_j^k$, which implies
\begin{eqnarray*}
&&  g_l(y_l)-g_l(y_l^{k+1})  \\
&& + \left\langle  B_l (y_l-y_l^{k+1}), - \lambda^{k+\frac{1}{2}}+\sigma_2\beta B_l\left(y_l^{k+1}-y_l^k\right)
+ \beta ( \C{A}\m{x}^{k+1}+\C{B}\m{y}^{k+1} -c) \right\rangle \\
& & -  \beta \left \langle  B_l (y_l-y_l^{k+1}),  \sum\limits_{j=1,j\neq l}^{q}B_j (y_j^{k+1} - y_j^k) \right\rangle \ge 0.
\end{eqnarray*}
For $l=1,2,\cdots,q$, letting $y_l=y_l^k$ in the above inequality and summing them together, we can deduce that
\begin{eqnarray} \label{optim-gyk}
&& \sum\limits_{l=1}^{q}\left(g_l(y_l^k)-g_l(y_l^{k+1}) \right)
 + \left\langle \C{B} (\m{y}^k-\m{y}^{k+1}), - \lambda^{k+\frac{1}{2}}+\beta
 \left(\mathcal{A}\m{x}^{k+1}+\mathcal{B}\m{y}^{k+1}-c\right)\right\rangle \\
&\ge& \left\|\m{y}^{k+1}-\m{y}^k\right\|^2_{H_\mathbf{y}}, \nonumber
\end{eqnarray}
where
\begin{eqnarray}\label{def-Hy}
H_{\mathbf{y}}&=&\beta\left[\begin{array}{ccccccccccccc}
\sigma_2 B_1\tr B_1 &&&& - B_1\tr B_2        &&&& \cdots &&&&- B_1\tr B_{q}  \\
           &&& &                &&&&  &&&&    \\
- B_2\tr B_1        &&&& \sigma_2 B_2\tr B_2 &&&& \cdots &&&&- B_2\tr B_{q}  \\
&&& &                &&&&  &&&&    \\
\vdots                &&& & \vdots                 &&&&\ddots  &&&&\vdots   \\
&&& &                &&&&  &&&&    \\
 - B_q\tr B_1 &&&&  - B_q\tr B_2 &&&&\cdots &&&&\sigma_2 B_q\tr B_q
\end{array}\right]\\
&=&\beta\left[\begin{array}{cccc}
B_1 &  & & \\
  &  &\ddots & \\
  &  &  &B_q
\end{array}\right]\tr H_{\mathbf{y},0} \left[\begin{array}{cccc}
B_1 &  & & \\
  &  &\ddots & \\
  &  &  &B_q
\end{array}\right] \nonumber
\end{eqnarray}
and $H_{\mathbf{y},0}$ is defined in (\ref{titl-EHy0}).
Similarly, it follows from the optimality condition of $y_l^k$-subproblem  that
\begin{eqnarray*}
&&  g_l(y_l)-g_l(y_l^k)   + \left\langle  B_l (y_l-y_l^k),-  \lambda^{k-\frac{1}{2}}+\sigma_2\beta
B_l\left(y_l^k-y_l^{k-1}\right)
+ \beta (\C{A}\m{x}^k+\C{B}\m{y}^k -c) \right\rangle \\
& & -  \beta \left \langle  B_l (y_l-y_l^k),  \sum\limits_{j=1,j\neq l}^{q}
B_j (y_j^k - y_j^{k-1}) \right\rangle \ge 0.
\end{eqnarray*}
For $l=1,2,\cdots,q$, letting $y_l=y_l^{k+1}$ in the above inequality and summing them together, we obtain
\begin{eqnarray}\label{optim-gykk}
& & \sum\limits_{l=1}^{q}\left(g_l(y_l^{k+1})-g_p(y_l^k) \right)  + \left\langle \C{B} (\m{y}^{k+1}-\m{y}^k), - \lambda^{k-\frac{1}{2}}+\beta \left(\mathcal{A}\m{x}^{k}+\mathcal{B}\m{y}^{k}-c\right) \right\rangle \\
& \ge & (\m{y}^{k}-\m{y}^{k+1})\tr {H_\mathbf{y}}(\m{y}^{k}-\m{y}^{k-1}). \nonumber
\end{eqnarray}
Since $\sigma_2 \in (q-1, \infty)$ and all $B_j$, $j=1,\cdots, q$, have full column rank, we have from (\ref{def-Hy}) that $H_\mathbf{y}$ is positive definite. Meanwhile, by the Cauchy-Schwartz inequality, we get
\begin{eqnarray}\label{HyIeq}
  \left\|\m{y}^{k+1}-\m{y}^k\right\|^2_{H_\mathbf{y}}+\left(\m{y}^{k}-\m{y}^{k+1}\right)\tr {H_\mathbf{y}}\left(\m{y}^{k}-\m{y}^{k-1}\right)
 \ge  \frac{1}{2}\left(\left\|\m{y}^{k+1}-\m{y}^k\right\|^2_{H_\mathbf{y}}-\left\|\m{y}^{k}-\m{y}^{k-1}\right\|^2_{H_\mathbf{y}}\right).
\end{eqnarray}
By adding (\ref{optim-gyk}) to (\ref{optim-gykk}) and using (\ref{HyIeq}), we achieve
\begin{eqnarray}\label{Ieq-yB}
&&  \left\langle \C{B} (\m{y}^{k}-\m{y}^{k+1}), \lambda^{k-\frac{1}{2}}-\lambda^{k+\frac{1}{2}}
+\beta (\mathcal{A}\m{x}^{k+1}+\mathcal{B}\m{y}^{k+1}-c ) \right\rangle \\
& \ge&   \langle \C{B} (\m{y}^{k}-\m{y}^{k+1}), \beta( \mathcal{A}\m{x}^{k}+\mathcal{B}\m{y}^{k}-c)\rangle
+ \frac{1}{2}\left(\left\|\m{y}^{k+1}-\m{y}^k\right\|^2_{H_\mathbf{y}}-\left\|\m{y}^{k}-\m{y}^{k-1}\right\|^2_{H_\mathbf{y}}\right). \nonumber
\end{eqnarray}
From the update of $\lambda^{k+\frac{1}{2}}$, i.e.,
$\lambda^{k+\frac{1}{2}}=\lambda^k-\tau\beta\left(\mathcal{A}\m{x}^{k+1}+\mathcal{B}\m{y}^{k}-c\right)$
and the update of $\lambda^k$, i.e.,
$\lambda^k=\lambda^{k-\frac{1}{2}}-s\beta\left(\mathcal{A}\m{x}^{k}+\mathcal{B}\m{y}^{k}-c\right),$
we have
\begin{eqnarray*}
 \lambda^{k-\frac{1}{2}}-\lambda^{k+\frac{1}{2}}
 = \tau\beta (\mathcal{A}\m{x}^{k+1}+\mathcal{B}\m{y}^{k+1}-c)
+s\beta (\mathcal{A}\m{x}^{k}+\mathcal{B}\m{y}^{k}-c)+
\tau\beta\mathcal{B} (\m{y}^k-\m{y}^{k+1}).
\end{eqnarray*}
Substituting the above inequality into the left-term of (\ref{Ieq-yB}), the proof is completed. $\ \ \ \diamondsuit$
\end{proof}

\begin{theorem} \label{Thoe-wGnorm}
Suppose $\sigma_1\in (p-1,+\infty)$, $\sigma_2\in (q-1,+\infty)$ and $\tau > -1$.
For the sequences $\{\m{w}^k\}$ and $\{ \widetilde{\m{w}}^k \}$
generated by GS-ADMM , there exists $\xi_1>0$ such that
\begin{eqnarray} \label{lbd-123}
&&  \left\|\m{w}^k-\widetilde{\m{w}}^k\right\|_G^2\\
&\ge &   \xi_1\left(\sum\limits_{i=1}^{p}\left\|A_i\left(x_i^k-x_i^{k+1}\right)\right\|^2
+\sum\limits_{j=1}^{q}\left\|B_j\left(y_j^k-y_j^{k+1}\right)\right\|^2\right) \nonumber \\
& + & (2-\tau-s)\beta\left\|\mathcal{A}\m{x}^{k+1}+\mathcal{B}\m{y}^{k+1}-c\right\|^2
+\frac{1-\tau}{1+\tau}\left(\left\|\m{y}^{k+1}-\m{y}^k\right\|^2_{H_\mathbf{y}}-\left\|\m{y}^{k}-\m{y}^{k-1}\right\|^2_{H_\mathbf{y}}\right) \nonumber \\
& + & \frac{(1-\tau)^2}{1+\tau}\beta\left\|\mathcal{B}\left(\m{y}^k-\m{y}^{k+1}\right)\right\|^2 + \frac{2(1-\tau)(1-s)}{1+\tau}\beta\left(\mathcal{A}\m{x}^{k}+\mathcal{B}\m{y}^{k}-c\right)\tr\mathcal{B}\left(\m{y}^k-\m{y}^{k+1}\right). \nonumber
\end{eqnarray}
\end{theorem}
\begin{proof} The inequality (\ref{lbd-123}) is directly obtained from (\ref{lower-wG1})  and (\ref{Ieq-ABw}). $\ \ \ \diamondsuit$
\end{proof}

The following theorem gives another variant of the lower bound of
$\|\m{w}^k-\widetilde{\m{w}}^k\|_G^2$, which plays a key role in showing the convergence of GS-ADMM.
\begin{theorem} \label{Thoe-38}
 Let the sequences  $\{\m{w}^k\}$ and $\{ \widetilde{\m{w}}^k \}$  be generated by GS-ADMM.
Then, for any
\begin{equation}\label{conv-region}
\sigma_1\in (p-1,+\infty),\quad \sigma_2\in (q-1,+\infty) \quad \mbox{and} \quad (\tau,s) \in \C{K},
 \end{equation}
where $\C{K}$ is defined in (\ref{setK}),
there exist constants $\xi_i(i=1,2)>0$ and $\xi_j(j=3,4)\ge0$, such that
\begin{eqnarray}\label{lbd-456}
 \left\|\m{w}^k-\widetilde{\m{w}}^k\right\|_G^2
& \ge &\xi_1\left(\sum\limits_{i=1}^{p}\left\|A_i\left(x_i^k-x_i^{k+1}\right)\right\|^2
+\sum\limits_{j=1}^{q}\left\|B_j\left(y_j^k-y_j^{k+1}\right)\right\|^2\right)  \\
&& + \xi_2\left\|\mathcal{A}\m{x}^{k+1}+\mathcal{B}\m{y}^{k+1}-c\right\|^2 \nonumber \\
&& + \xi_3\left(\left\|\mathcal{A}\m{x}^{k+1}+\mathcal{B}\m{y}^{k+1}-c\right\|^2-\left\|\mathcal{A}\m{x}^{k}+\mathcal{B}\m{y}^{k}-c\right\|^2\right) \nonumber\\
&& + \xi_4\left(\left\|\m{y}^{k+1}-\m{y}^k\right\|^2_{H_\mathbf{y}}-\left\|\m{y}^{k}-\m{y}^{k-1}\right\|^2_{H_\mathbf{y}}\right). \nonumber
\end{eqnarray}
\end{theorem}
\begin{proof} By the Cauchy-Schwartz inequality, we have
\begin{eqnarray} \label{Cau-Schw}
&& 2(1-\tau)(1-s)\left(\mathcal{A}\m{x}^{k}+\mathcal{B}\m{y}^{k}-c\right)\tr\mathcal{B}\left(\m{y}^k-\m{y}^{k+1}\right)\\
& \ge &
-(1-s)^2\left\|\mathcal{A}\m{x}^{k}+\mathcal{B}\m{y}^{k}-c\right\|^2-(1-\tau)^2\left\|\mathcal{B}\left(\m{y}^k-\m{y}^{k+1}\right)\right\|^2. \nonumber
\end{eqnarray}
Since
\[
 -\tau^2 - s^2 -\tau s + \tau + s + 1 = - \tau^2 + (1-s)(\tau +s) + 1,
\]
we have $\tau > -1$ when $(\tau, s) \in \C{K}$.
Then, combining (\ref{Cau-Schw}) with Theorem  \ref{Thoe-wGnorm}, we deduce
\begin{eqnarray}
\quad \left\|\m{w}^k-\widetilde{\m{w}}^k\right\|_G^2&\ge & \xi_1 \left(\sum\limits_{i=1}^{p}\left\|A_i\left(x_i^k-x_i^{k+1}\right)\right\|^2
+\sum\limits_{j=1}^{q}\left\|B_j\left(y_j^k-y_j^{k+1}\right)\right\|^2\right)\\
&& + \left(2-\tau-s-\frac{(1-s)^2}{1+\tau}\right)\beta\left\|\mathcal{A}\m{x}^{k+1}+\mathcal{B}\m{y}^{k+1}-c\right\|^2 \nonumber \\
&& +\frac{(1-s)^2}{1+\tau}\beta\left(\left\|\mathcal{A}\m{x}^{k+1}+\mathcal{B}\m{y}^{k+1}-c\right\|^2-\left\|\mathcal{A}\m{x}^{k}+\mathcal{B}\m{y}^{k}-c\right\|^2\right) \nonumber \\
&& + \frac{1-\tau}{1+\tau}\left(\left\|\m{y}^{k+1}-\m{y}^k\right\|^2_{H_\mathbf{y}}-\left\|\m{y}^{k}-\m{y}^{k-1}\right\|^2_{H_\mathbf{y}}\right), \nonumber
\end{eqnarray}
where $\xi_1 >0$ is the constant in Theorem  \ref{Thoe-wGnorm}.
Since $ -1 < \tau \le 1  $ and $ \beta > 0 $, we have
\begin{equation}\label{xi3-xi4}
\xi_3 := \frac{(1-s)^2}{1+\tau}\beta \ge 0 \quad \mbox{and} \quad \xi_4 := \frac{1-\tau}{1+\tau} \ge 0.
\end{equation}
In addition, when $(\tau, s) \in  \C{K} $, there is
\[
 -\tau^2 - s^2 -\tau s + \tau + s + 1 >0,
\]
which, by $\tau > -1$ and $\beta >0$, implies
\begin{equation}\label{xi2}
\xi_2 :=\left(2-\tau-s-\frac{(1-s)^2}{1+\tau}\right)\beta >0.
\end{equation}
Hence, the proof is completed. $\ \ \ \diamondsuit$
\end{proof}

\subsection{Global convergence}
In this subsection, we  show the global convergence  and the worst-case $O(1/t)$ convergence rate of GS-ADMM.
The following corollary is obtained directly from Theorems \ref{ther-32}-\ref{Ieq-33} and Theorem \ref{Thoe-38}.
\begin{corollary}
Let the sequences $\{\m{w}^k\}$ and $\{ \widetilde{\m{w}}^k \}$ be generated by GS-ADMM. For any
$(\sigma_1, \sigma_2, \tau, s)$ satisfying (\ref{conv-region}),
there exist constants $\xi_i(i=1,2)>0$ and $\xi_j(j=3,4)\ge0$  such that
\begin{eqnarray} \label{Final-wH}
 && \left\|\m{w}^{k+1}-\m{w}^*\right\|_H^2+ \xi_3\left\|\mathcal{A}\m{x}^{k+1}+\mathcal{B}\m{y}^{k+1}-c\right\|^2+\xi_4\left\|\m{y}^{k+1}-\m{y}^{k}\right\|^2_{H_\mathbf{y}}\\
& \le &  \left\|\m{w}^{k}-\m{w}^*\right\|_H^2+\xi_3\left\|\mathcal{A}\m{x}^{k}+\mathcal{B}\m{y}^{k}-c\right\|^2
+\xi_4\left\|\m{y}^{k}-\m{y}^{k-1}\right\|^2_{H_\mathbf{y}} \nonumber \\
& & - \xi_1\left(\sum\limits_{i=1}^{p}\|A_i\left(x_i^k-x_i^{k+1}\right)\|^2
+\sum\limits_{j=1}^{q}\left\|B_j\left(y_j^k-y_j^{k+1}\right)\right\|^2\right) \nonumber \\
& & - \xi_2\left\|\mathcal{A}\m{x}^{k+1}+\mathcal{B}\m{y}^{k+1}-c\right\|^2, \quad \forall \m{w}^*\in \mathcal{M}^*, \nonumber
\end{eqnarray}
and
\begin{eqnarray} \label{Final-fJ}
\quad && h(\m{u})-h(\widetilde{\m{u}}^k)+ \langle w-\widetilde{\m{w}}^k, \mathcal{J}(\m{w}) \rangle \\
& \ge & \frac{1}{2}\left( \|\m{w}-\m{w}^{k+1}\|_H^2+\xi_3\left\|\mathcal{A}\m{x}^{k+1}+\mathcal{B}\m{y}^{k+1}-c\right\|^2+\xi_4\left\|\m{y}^{k+1}-\m{y}^{k}\right\|^2_{H_\mathbf{y}}\right)
 \nonumber \\
&& - \frac{1}{2}\left( \|\m{w}-\m{w}^{k}\|_H^2+\xi_3\left\|\mathcal{A}\m{x}^{k}+\mathcal{B}\m{y}^{k}-c\right\|^2+\xi_4\left\|\m{y}^{k}-\m{y}^{k-1}\right\|^2_{H_\mathbf{y}}\right),
\quad \forall \m{w} \in \mathcal{M}. \nonumber
\end{eqnarray}
\end{corollary}


\begin{theorem} \label{Conver-P}
Let the sequences $\{\m{w}^k\}$ and $\{ \widetilde{\m{w}}^k \}$ be generated by GS-ADMM. Then, for any $(\sigma_1, \sigma_2, \tau, s)$ satisfying (\ref{conv-region}),
we have
\begin{equation} \label{xy-diff}
\lim_{k\rightarrow \infty} \left(\sum\limits_{i=1}^{p}\left\|A_i\left(x_i^k-x_i^{k+1}\right)\right\|^2
+\sum\limits_{j=1}^{q}\left\|B_j\left(y_j^k-y_j^{k+1}\right)\right\|^2
\right)=0,
\end{equation}
\begin{equation}\label{sat-fea}
\lim_{k\rightarrow \infty} \left\|\mathcal{A}\m{x}^k+\mathcal{B}\m{y}^k-c\right\| = 0,
\end{equation}
and there exists a $\m{w}^\infty\in \mathcal{M}^*$ such that
\begin{equation}\label{conv-w}
\lim_{k\rightarrow \infty} \widetilde{\m{w}}^{k} = \m{w}^\infty.
\end{equation}
\end{theorem}
\begin{proof} Summing the inequality (\ref{Final-wH}) over $k=1,2,\cdots,\infty$, we have
\begin{eqnarray*}
 && \xi_1 \sum\limits_{k=1}^{\infty}  \left(\sum\limits_{i=1}^{p}\left\|A_i\left(x_i^k-x_i^{k+1}\right)\right\|^2
 +\sum\limits_{j=1}^{q}\left\|B_j\left(y_j^k-y_j^{k+1}\right)\right\|^2\right) + \xi_2 \sum\limits_{k=1}^{\infty}  \left\|\mathcal{A}\m{x}^{k+1}+\mathcal{B}\m{y}^{k+1}-c\right\|^2 \\
&\le& \|\m{w}^{1}-\m{w}^*\|_H^2+\xi_1\left\|\mathcal{A}\m{x}^{1}+\mathcal{B}\m{y}^{1}-c\right\|^2+\xi_2\left\|\m{y}^{1}-\m{y}^{0}\right\|^2_{H_\mathbf{y}},
\end{eqnarray*}
which implies that (\ref{xy-diff}) and (\ref{sat-fea}) hold  since $\xi_1 >0$ and $\xi_2 >0$.

Because $(\sigma_1, \sigma_2, \tau, s)$ satisfy (\ref{conv-region}), we have by Lemma \ref{spd-H} that $H$ is positive definite.
So, it follows from  (\ref{Final-wH}) that the sequence $\{\m{w}^{k}\}$ is uniformly bounded. Therefore, there exits
a subsequence $\{\m{w}^{k_j}\}$ converging to a point $\m{w}^\infty = (\m{x}^\infty, \m{y}^\infty, \lambda^\infty) \in \C{M}$.
In addition, by the definitions of $\widetilde{x}_k$, $\widetilde{y}_k$ and $\widetilde{\lambda}_k$ in (\ref{tilde-xy}) and (\ref{tilde-lambda}),
it follows from (\ref{xy-diff}), (\ref{sat-fea}) and the full column rank assumption of all the matrices $A_i$ and $B_j$ that
\begin{equation}\label{xylambda-diff}
\lim_{k\rightarrow \infty} x_i^k-\widetilde{x}_i^k = \m{0}, \quad  \lim_{k\rightarrow \infty} y_j^k-\widetilde{y}_j^k = \m{0} \quad \mbox{and}
\quad \lim_{k\rightarrow \infty} \lambda^k-\widetilde{\lambda}^k = \m{0},
\end{equation}
for all $i=1, \cdots, p$ and $j=1, \cdots, q$. So, we have $\lim\limits_{k \to \infty} \m{w}^k - \widetilde{\m{w}}^k = {\bf 0}$.
Thus, for any fixed $\mathbf{w} \in \C{M}$, taking $\widetilde{\m{w}}^{k_j}$ in (\ref{app-vi}) and letting $j$ go to $\infty$, we obtain
\begin{equation} \label{Final-AB}
 h(\m{u})-h(\m{u}^\infty)+ \langle \mathbf{w}- \m{w}^\infty, \mathcal{J}(\m{w}^\infty) \rangle \ge 0.
\end{equation}
Hence, $\m{w}^\infty \in \C{M}^*$ is a solution point of $\textrm{VI}(h, \mathcal{J}, \mathcal{M})$ defined in (\ref{opt-vi}).

Since (\ref{Final-wH}) holds for any $\m{w}^* \in \C{M}^*$, by (\ref{Final-wH}) and $\m{w}^\infty \in \C{M}^*$, for all $l \ge k_j$, we have
\begin{eqnarray*}
&& \left\|\m{w}^{l}-\m{w}^\infty\right\|_H^2+ \xi_3\left\|\mathcal{A}\m{x}^{l}+\mathcal{B}\m{y}^{l}-c\right\|^2+\xi_4\left\|\m{y}^{l}-\m{y}^{l-1}\right\|^2_{H_\mathbf{y}}\\
& \le &  \left\|\m{w}^{k_j}-\m{w}^\infty\right\|_H^2+\xi_3\left\|\mathcal{A}\m{x}^{k_j}
+\mathcal{B}\m{y}^{k_j}-c\right\|^2+\xi_4\left\|\m{y}^{k_j}-\m{y}^{k_j-1}\right\|^2_{H_\mathbf{y}}.
\end{eqnarray*}
This together with (\ref{sat-fea}), (\ref{xylambda-diff}), $\lim\limits_{j \to \infty} \m{w}^{k_j} = \m{w}^\infty$
and the positive definiteness of $H$ illustrate $\lim\limits_{l \to \infty} \m{w}^l = \m{w}^\infty$.
Therefore, the whole sequence $\{\m{w}^k\}$ converges to the solution $\m{w}^\infty \in \C{M}^*$. This completes the whole proof. $\ \ \ \diamondsuit$
\end{proof}

The above Theorem \ref{Conver-P} shows the  global  convergence of our GS-ADMM. Next, we show the $\C{O}(1/t)$ convergence rate for
the ergodic iterates
\begin{equation} \label{ergodic-iterate}
\m{w}_t :=\frac{1}{t}\sum_{k=1}^{t}\widetilde{\m{w}}^{k} \quad \mbox{and}
\quad \m{u}_t :=\frac{1}{t}\sum_{k=1}^{t}\widetilde{\m{u}}^{k}.
\end{equation}
\begin{theorem} \label{Conv-rate}
Let the sequences $\{\m{w}^k\}$ and $\{ \widetilde{\m{w}}^k \}$ be generated by GS-ADMM. Then,  for any  $(\sigma_1, \sigma_2, \tau, s)$ satisfying (\ref{conv-region}),
there exist $\xi_j(j =3,4) \ge 0$ such that
\begin{eqnarray}\label{ut-conv}
&&  h(\m{u}_t)-h(\m{u})+ \langle \m{w}_t-\m{w}, \mathcal{J}(\m{w})\rangle\\
& \le & \frac{1}{2t}\left( \left\|\m{w}-\m{w}^{1}\right\|_H^2+\xi_3\left\|\mathcal{A}\m{x}^{1}+\mathcal{B}\m{y}^{1}-c\right\|^2
+\xi_4\left\|\m{y}^{1}-\m{y}^{0}\right\|^2_{H_\mathbf{y}}\right), \quad
 \forall \mathbf{w}\in \mathcal{M}. \nonumber
\end{eqnarray}
\end{theorem}
\begin{proof}
For $k=1, \cdots, t$, summing the inequality (\ref{Final-fJ}), we have
\begin{eqnarray} \label{Fin-123}
 && t h(\m{u}) - \sum\limits_{k=1}^{t}h(\widetilde{\m{u}}^k)+\left\langle t \m{w} -
 \sum\limits_{k=1}^{t} \widetilde{\m{w}}^k, \mathcal{J}(\m{w})\right\rangle\\
 & \ge & \frac{1}{2}\left( \left\|\m{w}-\m{w}^{t+1}\right\|_H^2+\xi_3 \left\|\mathcal{A}\m{x}^{t+1}+\mathcal{B}\m{y}^{t+1}-c\right\|^2+\xi_4\left\|\m{y}^{t+1}-\m{y}^{t}\right\|^2_{H_\mathbf{y}}\right)
  \nonumber \\
&& -  \frac{1}{2}\left( \|\m{w}-\m{w}^{1}\|_H^2+\xi_3\left\|\mathcal{A}\m{x}^{1}+\mathcal{B}\m{y}^{1}-c\right\|^2+\xi_4\left\|\m{y}^{1}-\m{y}^{0}\right\|^2_{H_\mathbf{y}}\right),
\quad \forall \m{w}\in \mathcal{M}. \nonumber
\end{eqnarray}
Since  $(\sigma_1, \sigma_2, \tau, s)$ satisfy (\ref{conv-region}),
 $H_\mathbf{y}$ is positive definite. And by Lemma \ref{spd-H}, $H$ is
also positive definite. So, it follows from (\ref{Fin-123}) that
\begin{eqnarray}\label{Fin-Ieq}
 && \frac{1}{t}\sum\limits_{k=1}^{t}h(\widetilde{\m{u}}^k)-h(\m{u})+\left\langle \frac{1}{t}\sum\limits_{k=1}^{t}\widetilde{\m{w}}^k-\m{w}, \mathcal{J}(\m{w})\right\rangle\\
 & \le &
\frac{1}{2t} ( \left\|\m{w}-\m{w}^{1}\right\|_H^2+\xi_3\left\|\mathcal{A}\m{x}^{1}+\mathcal{B}\m{y}^{1}-c\right\|^2
+\xi_4\left\|\m{y}^{1}-\m{y}^{0}\right\|^2_{H_\mathbf{y}} ), \quad
\forall \m{w} \in \mathcal{M}. \nonumber
\end{eqnarray}
By the convexity of $h$ and (\ref{ergodic-iterate}),
we have
\[
 h(\m{u}_t)\leq \frac{1}{t}\sum_{k=1}^{t} h(\widetilde{\m{u}}^{k}).
\]
Then, (\ref{ut-conv}) follows from (\ref{Fin-Ieq}). $\ \ \ \diamondsuit$
\end{proof}
\begin{remark}
In the above Theorem \ref{Conver-P} and Theorem \ref{Conv-rate}, we assume the parameters $(\sigma_1, \sigma_2, \tau, s)$ satisfy (\ref{conv-region}).
However, because of the symmetric role played by the $x$ and $y$ iterates in the GS-ADMM, substituting the index $k+1$ by $k$ for the $x$ and $\lambda$ iterates,
the GS-ADMM algorithm (\ref{GS-ADMM})  can be clearly written as
\begin{equation}\label{GS-ADMM-1}
\  \left \{\begin{array}{lll}
 \textrm{For}\ j=1,2,\cdots,q,\\
 \quad y_{j}^{k+1}=\arg\min\limits_{y_{j}\in\mathcal{Y}_{j}} \mathcal{L}_\beta(\m{x}^k,y_1^k,\cdots, y_j,\cdots,y_q^k,\lambda^{k-\frac{1}{2}}) + Q_j^k(y_j), \\
 \quad \textrm{where } Q_j^k(y_j) =
\frac{\sigma_2\beta}{2}\left\|B_{j}(y_{j}-y_{j}^k)\right\|^2,\\
\lambda^{k}=\lambda^{k-\frac{1}{2}}-s\beta(\mathcal{A}\m{x}^k+\mathcal{B}\m{y}^{k+1}-c) \\ \\
 \textrm{For}\ i=1,2,\cdots,p,\\
 \quad x_{i}^{k+1}=\arg\min\limits_{x_{i}\in\mathcal{X}_{i}} \mathcal{L}_\beta (x_1^k,\cdots, x_i,\cdots,x_p^k,\m{y}^{k+1},\lambda^k)
+P_i^k(x_i), \\
\quad \textrm{where } P_i^k(x_i) =
\frac{\sigma_1\beta}{2}\left\|A_{i}(x_{i}-x_{i}^k)\right\|^2,\\
 \lambda^{k+\frac{1}{2}}=\lambda^k-\tau\beta(\mathcal{A}\m{x}^{k+1}+\mathcal{B}\m{y}^{k+1}-c).
\end{array}\right.
\end{equation}
So, by applying Theorem \ref{Conver-P} and Theorem \ref{Conv-rate} on the algorithm (\ref{GS-ADMM-1}), it also converges and has the $\C{O}(1/t)$ convergence rate when
 $(\sigma_1, \sigma_2, \tau, s)$ satisfy
\begin{equation}\label{conv-region-2}
\sigma_1\in (p-1,+\infty),\quad \sigma_2\in (q-1,+\infty) \quad \mbox{and} \quad (\tau,s) \in \overline{\C{K}},
 \end{equation}
where
\begin{equation}\label{barK}
 \overline{\C{K}} = \left\{ (\tau, s) \ | \ \tau + s >0, \ s \le 1, \ -\tau^2 - s^2 -\tau s + \tau + s + 1 >0 \right\}.
\end{equation}
Hence, the convergence domain $\C{K}$ in Theorem \ref{Conver-P} and Theorem \ref{Conv-rate} can be easily enlarged to the symmetric domain, shown in Fig. 2,
\begin{equation}\label{setG}
\C{G}  = \C{K} \cup \overline{\C{K}}
     =  \left\{ (\tau, s) \ | \ \tau + s >0,\ -\tau^2 - s^2 -\tau s + \tau + s + 1 >0 \right\}.
\end{equation}
\end{remark}

\begin{figure}[htbp]\label{figG}
 \begin{minipage}{1\textwidth}
 \centering
\resizebox{11cm}{8.cm}{\includegraphics{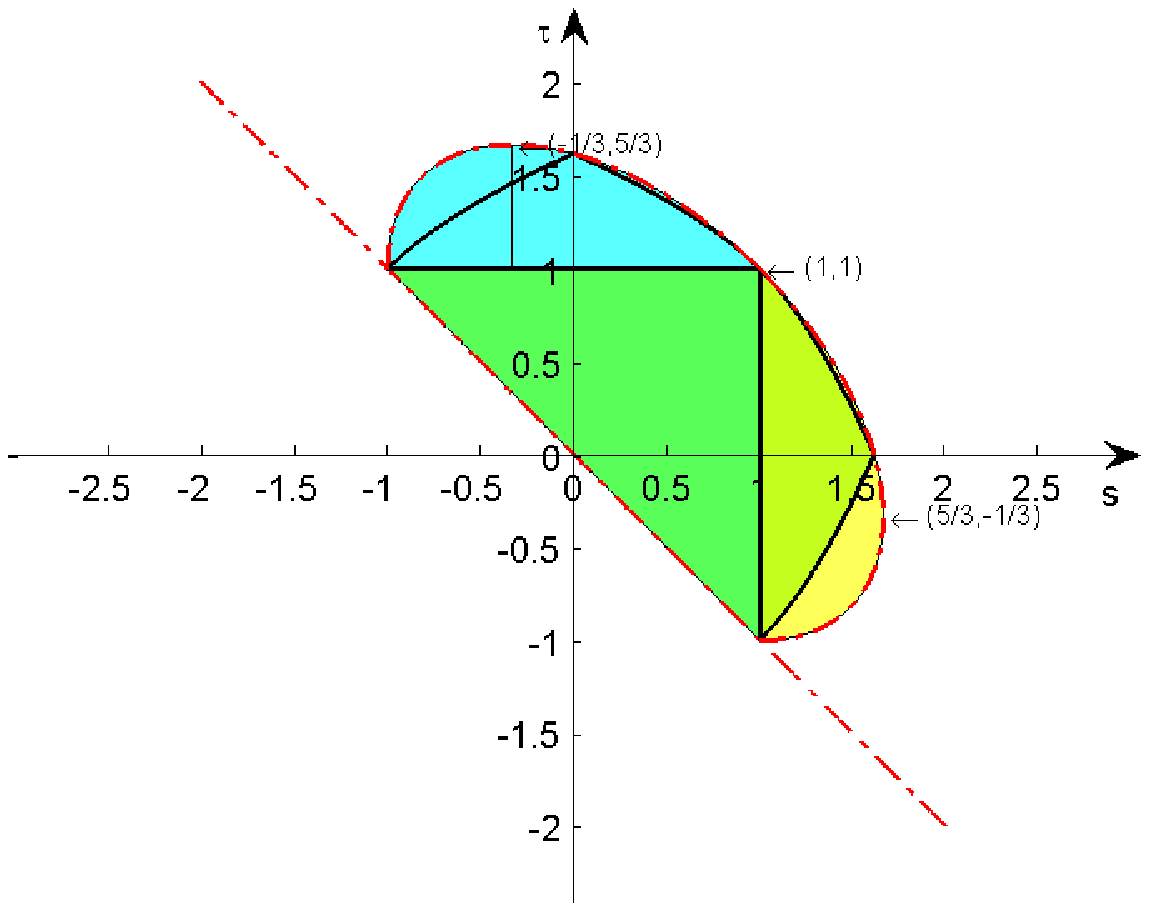}}
   \end{minipage}
\begin{center}
{\footnotesize Fig 2. Stepsize region $\mathcal{G}$ of GS-ADMM}
\end{center}
\end{figure}

\begin{remark}\label{rremar}
  Theorem \ref{Conver-P} implies that
we can   apply the following easily usable  stopping criterion for GS-ADMM:
\[
\max\left\{\left\|x_i^k-x_i^{k+1}\right\|_\infty,\left\|y_j^k-y_j^{k+1}\right\|_\infty, \left\|\mathcal{A}\m{x}^{k+1}+\mathcal{B}\m{y}^{k+1}-c\right\|_\infty\right\}\leq tol,
\]
where $tol >0$ is a given tolerance error.
One the other hand, Theorem \ref{Conv-rate} tells us that for any given compact set
$\mathcal{K}\subset \mathcal{M}$, if
\[
\eta = \sup_{\m{w}\in\mathcal{K}}\left\{ \|\mathbf{w}-\m{w}^{1}\|_H^2+\xi_3\left\|\mathcal{A}\m{x}^{1}+\mathcal{B}\m{y}^{1}-c\right\|^2+\xi_4\left\|\m{y}^{1}-\m{y}^{0}\right\|^2_{H_\mathbf{y}}\right\} < \infty,
 \]
we have
\[h(\m{u}_t)-h(\m{u})+ \left\langle \m{w}_t-\m{w}, \mathcal{J}(\m{w})\right\rangle \leq \frac{\eta}{2t},
\]
which shows that  GS-ADMM has a worst-case
$\mathcal{O}(1/t)$ convergence rate in an ergodic sense.
\end{remark}

\section{Two special cases of GS-ADMM}
Note that in GS-ADMM (\ref{GS-ADMM}), the two proximal parameters $\sigma_1$ and $\sigma_2$ are required to be strictly positive for the generalized
$p+q$ block separable convex programming. However, when taking $\sigma_1=\sigma_2=0$ for the two-block problem, i.e.,  $p=q=1$, GS-ADMM would reduce to
the scheme (\ref{he-admm}), which is globally convergent \cite{HeMaYuan2016}.
Such observations motivate us to  further investigate the following two special cases:
\[ \begin{array}{l}
\textrm{(a) GS-ADMM (\ref{GS-ADMM}) with}\ p\geq1, q=1, \sigma_1 \in (p-1, \infty)
\mbox{ and } \sigma_2=0;\\
\textrm{(b) GS-ADMM (\ref{GS-ADMM}) with}\ p=1, q\geq1, \sigma_1=0 \mbox{ and }
 \sigma_2 \in (q-1, \infty).
\end{array}
\]

\subsection{Convergence for the case (a)}
The corresponding GS-ADMM for the first case (a) can be simplified as follows:
\begin{equation}\label{GS-ADMM-a}
\quad\  \left \{\begin{array}{lll}
 \textrm{For}\ i=1,2,\cdots,p,\\
 \quad x_{i}^{k+1}=\arg\min\limits_{x_{i}\in\mathcal{X}_{i}} \mathcal{L}_\beta\left(x_1^k,\cdots, x_i,\cdots,x_p^k,y^k,\lambda^k\right)
+P_i^k(x_i), \\
\quad \textrm{where } P_i^k(x_i) =
\frac{\sigma_1\beta}{2}\left\|A_{i}\left(x_{i}-x_{i}^k\right)\right\|^2,\\
 \lambda^{k+\frac{1}{2}}=\lambda^k-\tau\beta\left(\mathcal{A}\m{x}^{k+1}+By^{k}-c\right),\\
y^{k+1}=\arg\min\limits_{y\in\mathcal{Y}} \mathcal{L}_\beta\left(\m{x}^{k+1},y,\lambda^{k+\frac{1}{2}}\right), \\
\lambda^{k+1}=\lambda^{k+\frac{1}{2}}-s\beta\left(\mathcal{A}\m{x}^{k+1}+By^{k+1}-c\right).
\end{array}\right.
\end{equation}
And, the corresponding matrices $Q,M,H$ and $G$ in this case are the following:
\begin{equation}\label{Q-a}
Q=\left[\begin{array}{cc|cc}
H_\mathbf{x} & & &{\bf 0}\\
\hline
{\bf 0} & & &\widetilde{Q}
\end{array}\right],
\end{equation}
where $H_\mathbf{x}$ is defined in (\ref{Hx}) and
\begin{equation}\label{tilde-Q-a}
\widetilde{Q} =\left[\begin{array}{c cc}
\beta B\tr B & \quad & -\tau B\tr \\
 -B & \quad & \frac{1}{\beta} I
\end{array}\right],
\end{equation}
\begin{equation}\label{M-a}
M=\left[\begin{array}{c|ccc}
I   &   &   \\ \hline
    & I  &  & \\
    & -s\beta B & \quad & (\tau+s)I \\
\end{array}\right],
 \end{equation}
\begin{equation}\label{H-a}
 H=QM^{-1}=\left[\begin{array}{c|cc c}
H_\mathbf{x} &   &  \\ \hline
&\left( 1- \frac{\tau s}{\tau+s} \right)\beta B\tr B & \quad &-\frac{\tau}{\tau+s} B\tr  \\
& & & \\
& -\frac{\tau}{\tau+s}B & \quad & \frac{1}{(\tau+s)\beta} I \\
\end{array}\right],
\end{equation}
\begin{equation}\label{G-a}
{\footnotesize
G=Q+Q\tr -M\tr HM=
\left[\begin{array}{c |cc c}
H_\mathbf{x}  &   &   \\ \hline
 &(1-s)\beta B\tr B & \quad & (s-1) B\tr \\
& & & \\
 & (s-1)B & \quad & \frac{2-\tau-s}{\beta} I  \\
\end{array}\right]. }\end{equation}
It can be verified that $H$ in (\ref{H-a}) is positive definite as long as
\[
\sigma_1\in(p-1,+\infty), \qquad (\tau,s) \in \{(\tau,s) \ | \ \tau +s>0, \tau<1 \}.
\]
Analogous to (\ref{eq-wGnorm}), we have
\begin{eqnarray} \label{wGnorm-a}
&& \quad \frac{1}{\beta} \|\m{w}^k-\widetilde{\m{w}}^k\|_G^2 \\
&=& \left\|\left(\begin{array}{l}
A_1\left(x_1^k-x_1^{k+1}\right)\\
A_2\left(x_2^k-x_2^{k+1}\right)\\
\ \ \ \ \ \vdots\\
A_{p}\left(x_{p}^k-x_{p}^{k+1}\right)
\end{array}\right)
\right\|_{  H_{\mathbf{x},0}}^2 + (2-\tau-s)\left\|\mathcal{A}\m{x}^{k+1}+B y^{k+1}-c\right\|^2 \nonumber\\
&  &  +(1-\tau)\left\|B\left(y^k
-y^{k+1}\right)\right\|^2+ 2(1-\tau) \left(\mathcal{A}\m{x}^{k+1}+By^{k+1}-c\right)\tr B\left(y^k
-y^{k+1}\right) \nonumber \\
&\ge & \xi_1 \sum\limits_{i=1}^{p}\left\|A_i\left(x_i^k-x_i^{k+1}\right)\right\|^2+(2-\tau-s)\left\|\mathcal{A}\m{x}^{k+1}+By^{k+1}-c\right\|^2 \nonumber \\
&&  +(1-\tau) \left\|B\left(y^k
-y^{k+1}\right)\right\|^2+ 2(1-\tau) \left(\mathcal{A}\m{x}^{k+1}+B y^{k+1}-c\right)\tr B\left(y^k
-y^{k+1}\right), \nonumber
\end{eqnarray}
for some $\xi_1 >0$, since $H_{\mathbf{x},0}$ defined in (\ref{Hx0}) is positive definite.
When $\sigma_2 = 0$, by a slight modification for the proof of Lemma \ref{Lem-36}, we
have the following lemma.
\begin{lemma} \label{Lem-36-a}
Suppose $\tau>-1$. The sequence $\{\m{w}^k\}$ generated by the algorithm (\ref{GS-ADMM-a}) satisfies
\begin{eqnarray}
&& \left(\mathcal{A}\m{x}^{k+1}+By^{k+1}-c\right)\tr B\left(y^k-y^{k+1}\right)  \nonumber\\
& \ge & \frac{1-s}{1+\tau}\left(\mathcal{A}\m{x}^{k}+By^{k}-c\right)\tr B\left(y^k-y^{k+1}\right)
 -\frac{\tau}{1+\tau}\left\| B\left(y^k-y^{k+1}\right)\right\|^2. \nonumber
\end{eqnarray}
\end{lemma}
Then, the following two theorems are similar to Theorem \ref{Thoe-38} and Theorem \ref{Conver-P}.
\begin{theorem} \label{Thoe-38-a}
 Let the sequences  $\{\m{w}^k\}$ and $\{ \widetilde{\m{w}}^k \}$  be generated by the algorithm (\ref{GS-ADMM-a}).
For any
\[
\sigma_1\in (p-1,+\infty) \quad \mbox{and} \quad (\tau,s) \in \C{K},
 \]
where $\C{K}$ is given in (\ref{setK}),
there exist constants $\xi_i(i=1,2)>0$ and $\xi_3\ge 0$, such that
\begin{eqnarray}\label{lbd-456-a}
 \left\|\m{w}^k-\widetilde{\m{w}}^k\right\|_G^2
& \ge &\xi_1 \sum\limits_{i=1}^{p}\|A_i\left(x_i^k-x_i^{k+1}\right)\|^2
 + \xi_2\left\|\mathcal{A}\m{x}^{k+1}+ By^{k+1}-c\right\|^2  \\
&& + \xi_3\left(\left\|\mathcal{A}\m{x}^{k+1}+ By^{k+1}-c\right\|^2-\left\|\mathcal{A}\m{x}^{k}+ By^{k}-c\right\|^2\right) \nonumber.
\end{eqnarray}
\end{theorem}
\begin{proof}
For any $(\tau, s) \in \C{K}$, we have $\tau > -1$. Then, by
Lemma \ref{Lem-36-a}, the inequality (\ref{wGnorm-a}) and the Cauchy-Schwartz inequality
(\ref{Cau-Schw}), we can deduce that (\ref{lbd-456-a}) holds
with $\xi_1$ given in (\ref{wGnorm-a}),
$\xi_2$ and $\xi_3$ given in (\ref{xi2}) and (\ref{xi3-xi4}), respectively. $\ \ \ \diamondsuit$
\end{proof}

\begin{theorem} \label{Conver-P-a}
Let the sequences $\{\m{w}^k\}$ and $\{ \widetilde{\m{w}}^k \}$ be generated by the algorithm (\ref{GS-ADMM-a}). For any
\begin{equation}\label{conv-region-1-a}
\sigma_1\in (p-1,+\infty) \quad \mbox{and} \quad (\tau,s) \in \C{K}_1,
 \end{equation}
where
\[
\C{K}_1 = \left\{(\tau, s) \ | \
\tau + s >0, \ \tau <1, \ -\tau^2 - s^2 -\tau s + \tau + s + 1 >0 \right\},
\]
we have
\begin{equation} \label{x-fea-a}
\lim_{k\rightarrow \infty} \sum\limits_{i=1}^{p}\left\|A_i\left(x_i^k-x_i^{k+1}\right)\right\|^2=0 \quad \mbox{and}
\quad \lim_{k\rightarrow \infty} \left\|\mathcal{A}\m{x}^k+By^k-c\right\| = 0,
\end{equation}
and there exists a $\m{w}^\infty\in \mathcal{M}^*$ such that
$
\lim\limits_{k\rightarrow \infty} \widetilde{\m{w}}^{k} = \m{w}^\infty.
$
\end{theorem}
\begin{proof}
First, it is clear that Theorem \ref{Ieq-33} still holds, which, combining with Theorem \ref{Thoe-38-a},
gives
\begin{eqnarray} \label{Final-wH-a}
 && \|\m{w}^{k+1}-\m{w}^*\|_H^2+ \xi_3\left\|\mathcal{A}\m{x}^{k+1}+By^{k+1}-c\right\|^2\\
& \le &  \|\m{w}^{k}-\m{w}^*\|_H^2+\xi_3\left\|\mathcal{A}\m{x}^{k}+By^{k}-c\right\|^2 \nonumber \\
& & - \xi_1 \sum\limits_{i=1}^{p}\|A_i\left(x_i^k-x_i^{k+1}\right)\|^2
 - \xi_2\left\|\mathcal{A}\m{x}^{k+1}+ By^{k+1}-c\right\|^2, \quad \forall \m{w}^*\in \mathcal{M}^*. \nonumber
\end{eqnarray}
Hence, (\ref{x-fea-a}) holds. For $(\sigma_1, \tau, s)$ satisfying (\ref{conv-region-1-a}), $H$ in (\ref{H-a})
is positive definite. So, by (\ref{Final-wH-a}), $\{\m{w}^k\}$ is uniformly bounded and therefore,
there exits a subsequence $\{\m{w}^{k_j}\}$ converging to a point $\m{w}^\infty = (\m{x}^\infty, y^\infty, \lambda^\infty) \in \C{M}$.
So, it follows from (\ref{x-fea-a}) and the full column rank assumption of all the matrices $A_i$ that
\begin{equation}\label{xlambda-diff-a}
\lim_{k\rightarrow \infty} x_i^k-\widetilde{x}_i^k = \lim_{k\rightarrow \infty} x_i^k-x_i^{k+1} = \m{0} \quad \mbox{and} \quad
\lim_{k\rightarrow \infty} \lambda^k-\widetilde{\lambda}^k = \m{0},
\end{equation}
for all $i=1, \cdots, p$. Hence, by $\lim\limits_{j \to \infty} \m{w}^{k_j} = \m{w}^\infty$ and (\ref{x-fea-a}), we have
\[
\lim_{j \to \infty} \m{x}^{k_j+1} = \m{x}^\infty \quad
\mbox{and} \quad \mathcal{A}\m{x}^\infty+By^\infty-c = \m{0},
\]
and therefore, by the full column rank assumption of $B$ and (\ref{x-fea-a}),
\[
\lim_{j \to \infty} y^{k_j+1} = \lim\limits_{j \to \infty} \widetilde{y}^{k_j} = y^\infty.
\]
Hence, by (\ref{xlambda-diff-a}), we have $ \lim\limits_{j \to \infty} \m{w}^{k_j} - \widetilde{\m{w}}^{k_j} = \m{0}$.
Thus, by taking $\widetilde{\m{w}}^{k_j}$ in (\ref{app-vi}) and letting $j$ go to $\infty$, the inequality  (\ref{Final-AB}) still holds.
Then, the rest proof of this theorem follows from the same proof of Theorem \ref{Conver-P}. $\ \ \ \diamondsuit$
\end{proof}

Based on Theorem \ref{Thoe-38-a} and by a similar proof to those of Theorem \ref{Conv-rate}, we can also show that
the algorithm (\ref{GS-ADMM-a}) has the worst-case  $\C{O}(1/t)$ convergence rate, which is omitted here for conciseness.

\subsection{Convergence for the case (b)}
The corresponding GS-ADMM for the case (b) can be simplified as follows
\begin{equation}\label{GS-ADMM-b}
\  \left \{\begin{array}{lll}
 x^{k+1}=\arg\min\limits_{x\in\mathcal{X}} \mathcal{L}_\beta (x,\m{y}^k,\lambda^k) \\
 \lambda^{k+\frac{1}{2}}=\lambda^k-\tau\beta(Ax^{k+1}+\mathcal{B}\m{y}^{k}-c),\\
 \textrm{For}\ j=1,2,\cdots,q,\\
 \quad y_{j}^{k+1}=\arg\min\limits_{y_{j}\in\mathcal{Y}_{j}} \mathcal{L}_\beta(x^{k+1},y_1^k,\cdots, y_j,\cdots,y_q^k,\lambda^{k+\frac{1}{2}}) + Q_j^k(y_j), \\
 \quad \textrm{where } Q_j^k(y_j) =
\frac{\sigma_2\beta}{2}\left\|B_{j}(y_{j}-y_{j}^k)\right\|^2,\\
\lambda^{k+1}=\lambda^{k+\frac{1}{2}}-s\beta(Ax^{k+1}+\mathcal{B}\m{y}^{k+1}-c).
\end{array}\right.
\end{equation}
In this case, the corresponding matrices $Q,M, H$ and $G$ become $\widetilde{Q},\widetilde{M}, \widetilde{H}$ and $\widetilde{G}$,
which are defined in (\ref{tilde-Q}), the lower-right block of $M$ in (\ref{M}), (\ref{tilde-H}) and (\ref{tilde-G}), respectively.

In what follows, let us define
\[
\m{v}^k=\left(\begin{array}{c}
\m{y}^k\\  \lambda^k \\
\end{array}\right) \quad \mbox{and} \quad
\widetilde{\m{v}}^k=\left(\begin{array}{c}
\widetilde{\m{y}}^k\\ \widetilde{\lambda}^k\\
\end{array}\right).
\]
Then, by the proof of Theorem \ref{Thoe-38}, we can deduce the following theorem.
\begin{theorem} \label{Thoe-38-b}
 Let the sequences  $\{\mathbf{v}^k\}$ and $\{ \widetilde{\mathbf{v}}^k\}$  be generated by the algorithm (\ref{GS-ADMM-b}).
For any
\[
\sigma_2 \in (q-1,+\infty) \quad \mbox{and} \quad (\tau,s) \in \C{K},
\]
where $\C{K}$ is defined in (\ref{setK}),
there exist constants $\xi_i(i=1,2)>0$ and $\xi_j(j=3,4)\ge0$ such that
\begin{eqnarray}\label{lbd-456-b}
 \left\|\m{v}^k-\widetilde{\m{v}}^k\right\|_{\widetilde{G}}^2
& \ge &\xi_1 \sum\limits_{j=1}^{q}\left\|B_j\left(y_j^k-y_j^{k+1}\right)\right\|^2  + \xi_2\left\|Ax^{k+1}+\mathcal{B}\m{y}^{k+1}-c\right\|^2 \nonumber\\
&& + \xi_3\left(\left\|Ax^{k+1}+\mathcal{B}\m{y}^{k+1}-c\right\|^2-\left\|Ax^{k}+\mathcal{B}\m{y}^{k}-c\right\|^2\right) \nonumber\\
&& + \xi_4\left(\left\|\m{y}^{k+1}-\m{y}^k\right\|^2_{H_\mathbf{y}}-\left\|\m{y}^{k}-\m{y}^{k-1}\right\|^2_{H_\mathbf{y}}\right). \nonumber
\end{eqnarray}
\end{theorem}

By slight modifications of  the proof of Theorem \ref{Conver-P} and Theorem \ref{Conver-P-a},
we have the following global convergence theorem.
\begin{theorem} \label{Conver-P-b}
Let the sequences $\{\m{w}^k\}$ and $\{ \widetilde{\m{w}}^k \}$ be generated by the algorithm (\ref{GS-ADMM-a}). Then, for any
\[
\sigma_2 \in (q-1,+\infty) \quad \mbox{and} \quad (\tau,s) \in \C{K},
 \]
where $\C{K}$ is defined in (\ref{setK}),
we have
\[
\lim_{k\rightarrow \infty} \sum\limits_{j=1}^{q}\left\|B_j\left(y_j^k-y_j^{k+1}\right)\right\|^2=0 \quad \mbox{and}
\quad \lim_{k\rightarrow \infty} \left\|Ax^k+\C{B}\m{y}^k-c\right\| = 0,
\]
and there exists a $\m{w}^\infty\in \mathcal{M}^*$ such that
$
\lim_{k\rightarrow \infty} \widetilde{\m{w}}^{k} = \m{w}^\infty.
$
\end{theorem}

By a similar proof to that of Theorem \ref{Conv-rate},
the algorithm (\ref{GS-ADMM-b}) also has the worst-case  $\C{O}(1/t)$ convergence rate.

\begin{remark}
Again, substituting the index $k+1$ by $k$ for the $x$ and $\lambda$ iterates, the algorithm (\ref{GS-ADMM-a}) can be also written as
\[
\  \left \{\begin{array}{lll}
y^{k+1}=\arg\min\limits_{y\in\mathcal{Y}} \mathcal{L}_\beta(\m{x}^k,y,\lambda^{k-\frac{1}{2}}), \\
\lambda^{k}=\lambda^{k-\frac{1}{2}}-s\beta(\mathcal{A}\m{x}^k+By^{k+1}-c) \\
 \textrm{For}\ i=1,2,\cdots,p,\\
 \quad x_{i}^{k+1}=\arg\min\limits_{x_{i}\in\mathcal{X}_{i}} \mathcal{L}_\beta (x_1^k,\cdots,, x_i,\cdots,x_p^k,y^{k+1},\lambda^k)
+P_i^k(x_i), \\
\quad \textrm{where } P_i^k(x_i) =
\frac{\sigma_1\beta}{2}\left\|A_{i}(x_{i}-x_{i}^k)\right\|^2,\\
 \lambda^{k+\frac{1}{2}}=\lambda^k-\tau\beta(\mathcal{A}\m{x}^{k+1}+B y^{k+1}-c).
\end{array}\right.
\]
So, by applying Theorem \ref{Conver-P-b} on the above algorithm, we know the algorithm (\ref{GS-ADMM-a})
also converges globally when
 $(\sigma_1, \tau, s)$ satisfy
\[
\sigma_1\in (p-1,+\infty), \quad \mbox{and} \quad (\tau,s) \in \overline{\C{K}},
 \]
where $\overline{\C{K}}$ is given in (\ref{barK}).
Hence, the convergence domain $\C{K}_1$ in Theorem \ref{Conver-P-a} can be enlarged to the symmetric domain
$\C{G} = \C{K}_1 \cup \overline{\C{K}}$ given in (\ref{setG}).
By a similar reason, the convergence domain $\C{K}$ in Theorem \ref{Conver-P-b} can be enlarged to
$\C{G}$ as well.
\end{remark}

\section{Numerical experiments}
In this section, we investigate the performance of the proposed  GS-ADMM
for solving a class of sparse matrix minimization problems.
All the algorithms are coded and simulated in MATLAB 7.10(R2010a) on a PC with Intel Core i5 processor(3.3GHz) with 4 GB memory.
\subsection{Test problem}
Consider the following Latent Variable Gaussian Graphical Model
Selection (LVGGMS) problem arising in the  statistical learning \cite{Chandrasekaran2012,Ma2017}:
\begin{equation}\label{Sec5-prob}
\begin{array}{lll}
\min\limits_{X,S,L\in\mathcal{R}^{n\times n}}  & F(X,S,L):=\langle X, C\rangle-\log\det(X)+ \nu\|S\|_1+\mu tr(L) \\
 \textrm{s.t. } &   X-S+L=\textbf{0},\ L\succeq\textbf{0},
\end{array}
\end{equation}
where $C\in\mathcal{R}^{n\times n}$ is the covariance matrix obtained from observation,
$\nu$ and $\mu$ are two given positive weight parameters,
 $tr(L)$ stands for the trace of the matrix $L$ and $\|S\|_1=\sum_{ij}|S_{ij}|$.
Clearly, by two different ways of partitioning the variables of (\ref{Sec5-prob}),
 the GS-ADMM (\ref{GS-ADMM}) can lead to the following two algorithms:
\begin{equation}\label{GS-ADMM-51}
 \left \{\begin{array}{lll}
X^{k+1}=\arg\min\limits_{X} \left\{\langle X,C\rangle-\log\det(X)+\frac{\beta}{2}\left\|X-S^k+L^k-\frac{\Lambda^{k}}{\beta}\right\|_F^2
+\frac{\sigma_1\beta}{2}\left\|X-X^k\right\|_F^2\right\}, \\
S^{k+1}=\arg\min\limits_{S} \left\{\nu\|S\|_1+\frac{\beta}{2}\left\|X^k-S+L^k-\frac{\Lambda^{k}}{\beta}\right\|_F^2
+\frac{\sigma_1\beta}{2}\left\|S-S^k\right\|_F^2\right\}, \\
\Lambda^{k+\frac{1}{2}}=\Lambda^{k}-\tau\beta(X^{k+1}-S^{k+1}+L^k), \\
L^{k+1}=\arg\min\limits_{L\succeq \textbf{0}} \left\{\mu tr(L)+\frac{\beta}{2}\left\|X^{k+1}-S^{k+1}+L-\frac{\Lambda^{k+\frac{1}{2}}}{\beta}\right\|_F^2
+\frac{\sigma_2\beta}{2}\left\|L-L^k\right\|_F^2\right\}, \\
 \Lambda^{k+1}=\Lambda^{k+\frac{1}{2}}-s\beta(X^{k+1}-S^{k+1}+L^{k+1});
\end{array}\right.
\end{equation}
\begin{equation}\label{GS-ADMM-522}
 \left \{\begin{array}{lll}
X^{k+1}=\arg\min\limits_{X} \left\{\langle X,C\rangle-\log\det(X)+\frac{\beta}{2}\left\|X-S^k+L^k-\frac{\Lambda^{k}}{\beta}\right\|_F^2
+\frac{\sigma_1\beta}{2}\left\|X-X^k\right\|_F^2\right\}, \\
\Lambda^{k+\frac{1}{2}}=\Lambda^{k}-\tau\beta(X^{k+1}-S^{k}+L^k), \\
S^{k+1}=\arg\min\limits_{S} \left\{\nu\|S\|_1+\frac{\beta}{2}\left\|X^{k+1}-S+L^k-\frac{\Lambda^{k+\frac{1}{2}}}{\beta}\right\|_F^2
+\frac{\sigma_2\beta}{2}\left\|S-S^k\right\|_F^2\right\}, \\
L^{k+1}=\arg\min\limits_{L\succeq \textbf{0}} \left\{\mu tr(L)+\frac{\beta}{2}\left\|X^{k+1}-S^{k}+L-\frac{\Lambda^{k+\frac{1}{2}}}{\beta}\right\|_F^2
+\frac{\sigma_2\beta}{2}\left\|L-L^k\right\|_F^2\right\}, \\
 \Lambda^{k+1}=\Lambda^{k+\frac{1}{2}}-s\beta(X^{k+1}-S^{k+1}+L^{k+1}).
\end{array}\right.
\end{equation}

Note that all the subproblems in (\ref{GS-ADMM-51}) and (\ref{GS-ADMM-522}) have
closed formula solutions. Next, we take the scheme (\ref{GS-ADMM-51}) for an example to show how to get the explicit solutions of the subproblem.
By the first-order optimality condition of the
$X$-subproblem in (\ref{GS-ADMM-51}), we derive
\[
\textbf{0}= C-X^{-1}+\beta\left(X-S^k+L^k-{\Lambda^{k}}/{\beta}\right)
+\sigma_1\beta(X-X^k)
\]
which is equivalent to
\[
 (\sigma_1+1)\beta X^2 +\left[C+\beta(L^k-S^k)-\Lambda^k-\sigma_1\beta X^k\right]X-\textbf{I}=\textbf{0}.
\]
Then, from the eigenvalue decomposition
\[
U\textrm{Diag}(\rho)U\tr = C+\beta(L^k-S^k)-\Lambda^k-\sigma_1\beta X^k,
\]
where $\textrm{Diag}(\rho)$ is a diagonal matrix with
$\rho_i, i=1, \cdots, n$, on the diagonal, we obtain that the solution of the $X$-subproblem in (\ref{GS-ADMM-51})
is
\[
X^{k+1}=U\textrm{Dia}g(\gamma)U \tr,
\]
where $\textrm{Diag}(\gamma)$ is the diagonal matrix with diagonal elements
\[
\gamma_i=\frac{-\rho_i+\sqrt{\rho_i^2+4(\sigma_1+1)\beta}}{2(\sigma_1+1)\beta}, \quad i=1,2,\cdots,n.
\]
On the other hand, the $S$-subproblem in (\ref{GS-ADMM-51}) is equivalent to
\[
\begin{array}{lll}
S^{k+1}&=&\arg\min\limits_{S} \left\{\nu\|S\|_1+
\frac{(\sigma_1+1)\beta}{2}\left\|S-\frac{X^k+L^k+\sigma_1S^k-\Lambda^k/\beta}{(\sigma_1+1)
}\right\|_F^2\right\}\\
&=& \textrm{Shrink}\left(\frac{X^k+L^k+\sigma_1S^k-\Lambda^k/\beta}{(\sigma_1+1)
}, \frac{\nu}{(\sigma_1+1)\beta}\right),
\end{array}
\]
where $\textrm{Shrink}(\cdot,\cdot)$ is the soft shrinkage operator (see e.g.\cite{TaoYuan2014}).
Meanwhile,  it is easy to verify that the  $L$-subproblem is equivalent to
\[
\begin{array}{lll}
L^{k+1}&=&\arg\min\limits_{L\succeq \textbf{0}}
\frac{(\sigma_2+1)\beta}{2}\left\|L-\widetilde{L}\right\|_F^2\\
&=& V\textrm{Diag}(\max\{\rho,\m{0}\})V\tr,
\end{array}
\]
where $\max\{\rho, \m{0}\}$ is taken component-wise and
$V\textrm{Diag}(\rho)V\tr$ is the eigenvalue decomposition of the matrix
\[
\widetilde{L}=\frac{\sigma_2 L^k+S^{k+1}+\Lambda^{k+\frac{1}{2}}/\beta-X^{k+1}-\mu \textbf{I}/\beta}{(\sigma_2+1)}.
\]

\subsection{Numerical results}
In the following, we investigate the performance of several algorithms for solving the LVGGMS problem, where all the corresponding subproblems can be solved
in a similar way as shown in the above analysis. For all algorithms, the maximal number of
iterations is set as $1000$, the starting iterative values are set as $(X^0,S^0, L^0,\Lambda^0)=(\textbf{I},\textbf{2I},\textbf{I},\textbf{0})$,
and motivated by Remark \ref{rremar}, the following stopping conditions are used
\begin{eqnarray*}
\textrm{IER(k)} &:= & \max\left\{\left\|X^k-X^{k-1}\right\|_\infty,\left\|S^k-S^{k-1}\right\|_\infty,  \left\|L^k-L^{k-1}\right\|_\infty\right\}\leq \textrm{TOL}, \\
\textrm{OER(k)} &:= & \frac{|F(X^{k},S^{k},L^{k})-F^*|}{|F^*|}\leq \textrm{Tol},
\end{eqnarray*}
together with $\textrm{CER(k)}:=\left\|X^{k}-S^{k}+L^{k}\right\|_F\leq10^{-4}$,
where $F^*$ is the approximate optimal objective function value obtained by running GS-ADMM (\ref{GS-ADMM-51}) after 1000 iterations.
In (\ref{Sec5-prob}), we set $(\nu,\mu)=(0.005,0.05)$ and
the given data $C$ is randomly generated by the following MATLAB code with $m=100$, which are
downloaded from S. Boyd's homepage\footnote{http://web.stanford.edu/$\sim$boyd/papers/admm/covsel/covsel\_example.html}:
{
\[ \begin{array}{l}
\verb"randn('seed',0); rand('seed',0); n=m; N=10*n;" \\
\verb"Sinv=diag(abs(ones(n,1))); idx=randsample(n^2,0.001*n^2);"\\
\verb"Sinv(idx)=ones(numel(idx),1); Sinv=Sinv+Sinv';"\\
\verb"if min(eig(Sinv))<0" \\
\quad\verb" Sinv=Sinv+1.1*abs(min(eig(Sinv)))*eye(n);"\\
\verb"end"\\
\verb"S=inv(Sinv);"\\
\verb"D=mvnrnd(zeros(1,n),S,N); C=cov(D);"\\
\end{array}
\]
}

\subsubsection{Performance of different versions of GS-ADMM}
In the following, we denote
 \[ \begin{array}{l}
\textrm{GS-ADMM}\ (\ref{GS-ADMM-51})\ \textrm{by}\ \textrm{``\textbf{GS-ADMM-I}''};\\
\textrm{GS-ADMM}\ (\ref{GS-ADMM-522})\ \textrm{by}\ \textrm{``\textbf{GS-ADMM-II}''};\\
\textrm{GS-ADMM}\ (\ref{GS-ADMM-51})\ \textrm{with}\ \sigma_2=0\ \textrm{by}\ \textrm{``\textbf{GS-ADMM-III}''};\\
\textrm{GS-ADMM}\ (\ref{GS-ADMM-522})\ \textrm{with}\ \sigma_1=0\ \textrm{by}\ \textrm{``\textbf{GS-ADMM-IV}''}.\\
\end{array}
\]

\vskip1cm
{\small
\begin{center}
\begin{tabular}{|c|ccccc|}
\hline
\textbf{GS-ADMM-I}/$\beta$&  Iter(k)  & CPU(s) & CER  & IER & OER\\
\hline
0.5  &  1000    & 15.29     & 7.2116e-8      &   \emph{5.0083e-6}    & 3.2384e-10 \\
0.2  &493    &8.58      &  1.4886e-8    & 9.8980e-8   &  5.7847e-11 \\
0.1  & 254   & 4.24      &  1.6105e-8    &  9.7867e-8  &  5.6284e-11\\
0.08 &202   & 3.27      &  1.7112e-8    & 9.8657e-8  &   5.6063e-11\\
0.07 & 175     & 3.03     &   1.7548e-8 &    9.7091e-8 &  5.4426e-11\\
0.06 & 146     & 2.42    &  1.9200e-8 &   9.9841e-8 & 5.4499e-11\\
0.05 &   115  &  \textbf{1.84}  & 1.9174e-8 &  8.8302e-8&4.4919e-11\\
0.03 & \textbf{112} &  2.21      &  1.7788e-7   & 9.9591e-8  &  2.2472e-11\\
0.01 & 270   & 4.50       & 6.4349e-7   &9.9990e-8   &2.5969e-10 \\
0.006& 424  &  7.57   &1.0801e-6   &9.8883e-8    &  5.0542e-10\\
0.004& 604 &  10.74  & 1.6490e-6   &  9.9185e-8    &  8.7172e-10\\
\hline
\textbf{GS-ADMM-II}/$\beta$&  Iter(k)  & CPU(s) &  CER  &IER  &OER\\
\hline
0.5  & 1000    &  15.80    & 8.8857e-8      &   \emph{3.2511e-6}    & 4.0156e-10 \\
0.2  &603    &11.35      &  3.7706e-9    & 9.9070e-8   &  1.2204e-12 \\
0.1  & 312   &  4.93     &  6.0798e-9    &  9.9239e-8  &  2.3994e-12\\
0.08 &250   &4.40       &  7.1384e-9    & 9.6234e-8  &    2.8127e-12\\
0.07 & 217     &  3.42    &   8.2861e-9 &    9.8471e-8 &  3.1878e-12\\
0.06 &  183    &  3.09   &  9.7087e-8 &   9.8298e-8& 3.4898e-12\\
0.05 &  147  &  2.85      &1.1335e-8    & 9.1450e-8  & 3.3405e-12 \\
0.03 & \textbf{114} & \textbf{1.85}  & 1.5606e-7   & 9.1283e-8  &  1.9479e-11 \\
0.01 & 271 & 4.70      & 6.2003e-7   &9.6960e-8  &2.4594e-10 \\
0.006& 424  &7.38   &1.0774e-6   &9.8852e-8  &   5.0224e-10\\
0.004& 604 & 10.01    &  1.6461e-6   &  9.9114e-8  & 8.6812 e-10\\
\hline
\textbf{GS-ADMM-III}/$\beta$&  Iter(k)  & CPU(s) & CER  & IER & OER\\
\hline
0.5  & 579   & 9.36     & 1.2740e-8      &  9.9818e-8   & 5.2821e-11 \\
0.2  & 247   & 5.52     &  1.2043e-8    & 9.6354e-8   &  4.5217e-11 \\
0.1  & 125   & 2.14      &  1.1737e-8    &  9.5170e-8  &  3.6207e-11\\
0.08 & 97  & 1.55      &  1.2078e-8    & 9.7603e-8  &   2.8773e-11\\
0.07 & 82     & 1.36     &   1.1854e-8 &    9.5322e-8 &  1.6215e-11\\
0.06 & \textbf{69}     & \textbf{1.27}    &  1.2680e-8 &   8.2352e-8& 1.5087e-11\\
0.05 &  71   &  1.40     & 9.1560e-8    & 9.8745e-8  & 8.1869e-12 \\
0.03 & 110   &1.71     & 1.8118e-7    &9.4257e-8   & 2.7549e-11 \\
0.01 & 271   & 4.46    &6.3390e-7     &9.7803e-8   & 2.5210e-10 \\
0.006& 424   & 6.92    & 1.0856e-6    &  9.9123e-8 & 5.0717e-10 \\
0.004&  604  &10.11     &1.6527e-6     &9.9275e-8   & 8.7303e-10 \\
\hline
\textbf{GS-ADMM-IV}/$\beta$&  Iter(k)  & CPU(s) &  CER  &IER  &OER\\
\hline
0.5  &  1000    & 15.76     & 7.1259e-8      &  \emph{2.6323e-6}    & 6.9956e-12 \\
0.2  & 587   &  9.08    &  3.8200e-9    & 9.9214e-8   &  1.3291e-12 \\
0.1  & 304   & 4.80      &  6.0296e-9    &  9.6197e-8  &  2.4309e-12\\
0.08 & 243  & 4.91      &  7.2062e-9    & 9.4484e-8  &   2.8670e-12\\
0.07 & 211     & 3.25     &   8.1772e-9 &    9.4133e-8 &  3.1477e-12\\
0.06 &  177    & 2.81    &  9.9510e-9 &   9.6911e-8& 3.5342e-12\\
0.05 & 140   &3.07     & 1.3067e-8    & 9.9446e-8  & 3.6691e-12 \\
0.03 & \textbf{115}   & \textbf{1.80}    &  1.6886e-7   &9.5844e-8   &  2.1829e-11\\
0.01 & 271   & 4.67    & 6.2006e-7    &9.7151e-8   & 2.4927e-10 \\
0.006& 424   & 6.94    & 1.0758e-6    &9.8755e-8   & 5.0454e-10 \\
0.004& 604   & 10.21    & 1.6454e-6    & 9.9088e-8  & 8.6995e-10 \\
\hline
\end{tabular}
\end{center}
}
\begin{center}
Table 1:\ Numerical results of GS-ADMM-I, II, III and IV  with different $\beta$.
\end{center}
\vskip4mm

First, we would like to investigate the performance of the above different versions of GS-ADMM
for solving the LVGGMS problem  with variance of the penalty parameter $\beta$.
The results are reported in Table 1 with $\textrm{TOL}=\textrm{Tol}=1.0\times 10^{-7}$,
and $(\tau,s)=(0.8,1.17)$. For GS-ADMM-I and GS-ADMM-II, $(\sigma_1,\sigma_2)=(2,3)$.
Here,  ``Iter" and ``CPU" denote the iteration number and the CPU time in seconds, and
the bold letter indicates the best result of each algorithm. From Table 1, we can observe that:
\begin{itemize}
\item Both the iteration number and the CPU time of all the algorithms have a similar
changing pattern, which decreases originally  and then increases along with the decrease of the  value of  $\beta$.
\item  For the same value of $\beta$, the results of GS-ADMM-III are  slightly better than other algorithms
in terms of the iteration number, CPU time,  and the feasibility errors CER, IER and OER.
\item GS-ADMM-III with $\beta=0.5$ can terminate after 579 iterations to achieve the tolerance $10^{-7}$,
while the other algorithms with $\beta=0.5$ fail to achieve this tolerance within given number of iterations.
\end{itemize}
In general, the algorithm  (\ref{GS-ADMM-51}) with $\beta = 0.06$ performs better than other cases.
Hence, in the following experiments for GS-ADMM, we adapt GS-ADMM-III with default $\beta = 0.06$.
 Also note that $\sigma_2 =0$, which is not allowed by the algorithms discussed in \cite{Heyuan2015}.
\vskip4mm
{\small
\begin{center}
\begin{tabular}{|c|ccccc|}
\hline
$(\tau,s)$&  Iter(k)  & CPU(s) & CER  & IER & OER\\
\hline
(1, -0.8)     & 256   & 4.20    & 9.8084e-5   &  7.8786e-6  & 1.1298e-7  \\
(1, -0.6)     & 144   & 2.39    & 5.7216e-5   &  9.9974e-6  & 3.8444e-8  \\\
(1, -0.4)     & 105   & 1.80    & 3.5144e-5   &  9.7960e-6  & 1.3946e-8  \\
(1, -0.2)     & 84   & 1.45    & 2.3513e-5   &  9.3160e-6  & 6.4220e-9  \\
(1,\ \ \ 0)    & 70   & 1.14    & 1.7899e-5   &  9.4261e-6  & 3.9922e-9  \\
(1, 0.2)      & 61   & 0.98    & 1.3141e-5   &  8.9191e-6  & 1.7780e-9  \\
(1, 0.4)      & 54   &  0.88   & 1.0549e-5   &  9.1564e-6  & 4.6063e-10  \\
(1, 0.6)      & 49   & 0.82    & 9.0317e-5   &  9.4051e-6  & 2.7938e-9  \\
(1, 0.8)      &\textbf{49}    &  \textbf{0.80}   & 3.5351e-5   &  8.0885e-6  & 1.4738e-9  \\
\hline
(-0.8, 1)     &229    & 3.91    & 9.9324e-5   &  8.4462e-6  & 1.9906e-7  \\
(-0.6, 1)     & 127   &2.06     & 6.1118e-5   &  9.6995e-6  & 7.8849e-8  \\
(-0.4, 1)     &  96  & 1.61    & 3.4111e-5   &  9.6829e-6  & 2.7549e-8  \\
(-0.2, 1)     & 79   & 1.30    & 2.2004e-5   &  9.6567e-6  & 1.2015e-8  \\
(0,\ \ \ 1)    & 67   & 1.16    & 1.6747e-5   &  9.9244e-6  & 6.2228e-9  \\
(0.2, 1)      & 59   & 0.93    & 1.2719e-5   &  9.4862e-6  & 2.9997e-9  \\
(0.4, 1)      &  53  &0.88     & 1.0253e-5   &  9.3461e-6  & 3.4811e-10  \\
(0.6, 1)      &  49 &  0.85   & 8.0343e-6   &  8.8412e-6  & 2.9837e-9  \\
(0.8, 1)      &  \textbf{49}   & \textbf{0.81}    & 3.3831e-6   &  8.1998e-6  & 2.1457e-9  \\
\hline
(1.6, -0.3)   & 60   &0.99     & 1.2111e-5   &  9.4583e-6  & 1.1705e-9  \\
(1.6, -0.6)   & 74   &1.22     & 1.8012e-5   &  9.6814e-6  & 2.7562e-9  \\
(1.5, -0.8)   & 97   & 1.68    & 3.1310e-5   &  9.8972e-6  & 1.4911e-8  \\
(1.3, 0.3)    & 50   &0.83     & 8.5476e-6   &  8.9655e-6  & 3.4389e-10  \\
(0.2, 0.5)    & 87   & 1.44    & 2.7160e-5   &  9.4503e-6  & 1.7906e-8  \\
(0.4, 0.9)    &  56  & 0.98    & 1.1060e-5   &  9.1081e-6  & 1.7179e-9  \\
(0.8, 1.17)   & 49   & 0.86    & 1.5419e-6   &  8.5023e-6  & 2.5529e-9  \\
(0, 1.618)   & 50   & 0.90    & 5.5019e-6   &  8.6980e-6  & 1.4722e-9  \\
(0.9, 1.09)   & \textbf{49}   & \textbf{0.78}    & 1.4874e-6   &  8.4766e-6  & 2.2194e-9  \\
\hline
(0.1, 0.1)   & 229    & 4.42    &  9.8698e-5   &   8.3622e-6  &  2.3575e-7  \\
(0.2, 0.2)   & 130    & 2.32    &  5.5559e-5   &   9.9888e-6  &  7.5859e-8  \\
(0.3, 0.3)   & 97    & 1.75    &  3.4344e-5   &   9.9362e-6  &  2.8190e-8  \\
(0.4, 0.4)   & 79    & 1.43    &  2.4256e-5   &   9.8539e-6  &  1.2790e-8  \\
(0.5, 0.5)   & 68    &1.15     &  1.6805e-5   &   9.2144e-6  &  5.5121e-9  \\
(0.6, 0.6)   & 59    & 0.98    &  1.3862e-5   &   9.7793e-6  &  2.8580e-9  \\
(0.7, 0.7)   & 53    & 0.91    &  1.1091e-5   &   9.6433e-6  &  3.9013e-12  \\
(0.8, 0.8)   &49     & 0.84    &  8.4235e-6   &   8.9432e-6  &  3.0519e-9  \\
(0.9, 0.9)   & \textbf{49}    &  \textbf{0.83}   &  3.4493e-6   &   8.1314e-6  &  1.8888e-9  \\
\hline
\end{tabular}
\end{center}
}
\begin{center}
Table 2:\ Numerical results of GS-ADMM-III  with different stepsizes $(\tau, s)$.
\end{center}

Second, we investigate how the stepsizes
$(\tau,s)\in \C{G}$ with different values would affect the performance of GS-ADMM-III.
 Table 2 reports the comparison results with variance of $(\tau,s)$ for
 $\textrm{TOL}=\textrm{Tol}=1.0\times 10^{-5}$.
 One obvious observation from Table 2 is that both the iteration number and the CPU time decrease along with the increase of s (or  $\tau$)
 for fixed value of $\tau$ (or s), which indicates that the stepsizes of $(\tau,s)\in \C{G}$
could influence the performance of  GS-ADMM significantly.
 In addition, the results in Table 2 also indicate that using more flexible but with both relatively
larger stepsizes $\tau$ and $s$ of the dual variables often gives the best convergence speed.
 Comparing all the reported results in Table 2,  by setting
$(\tau,s)=(0.9,1.09)$, GS-ADMM-III gives the relative best performance for solving
the problem (\ref{Sec5-prob}).

\subsubsection{Comparison of GS-ADMM with other state-of-the-art algorithms}
In this subsection, we would like to carry out some numerical comparison of solving the problem (\ref{Sec5-prob})
by using GS-ADMM-III and the other four methods:
\vskip4mm
{\begin{center}
{\small\begin{tabular}{|c|c|c|ccccc|}
\hline
 &TOL& Tol& Iter(k)  & CPU(s) & CER  & IER & OER\\
\hline
  &1e-3      &1e-7 &  \textbf{33}  &  \textbf{0.46} & 8.3280e-5   & 2.5770e-4 &  4.0973e-9\\
         &  &1e-12& \textbf{83}   & \textbf{1.16}  & 9.5004e-9  & 1.0413e-8 &  8.3240e-13\\
         &  & &     &    &   &   &   \\
\textbf{GS-ADMM-III}  &1e-6      &1e-8 & \textbf{58}   &  \textbf{0.84} & 8.3812e-7   & 9.0995e-7 &  7.6372e-11\\
    &  &1e-14& \textbf{108}   & \textbf{1.55}  &  1.0936e-10  & 1.2072e-10 &  9.5398e-15\\
     &  & &     &    &   &   &   \\
 &1e-9      &1e-7 & \textbf{97}   &\textbf{1.39}   & 7.7916e-10   & 8.5759e-10 &  6.8775e-14\\
  &  &1e-15&\textbf{118}   & \textbf{1.72}  &  1.8412e-11  & 2.0361e-11 &  6.6557e-16\\
   &  & &     &    &   &   &   \\
\hline
  &1e-3      &1e-7 & 62   & 0.88  & 9.6422e-5   &3.6934e-5 & 5.8126e-8\\
       &  &1e-12& 187   &2.74   &  9.4636e-9  & 3.4868e-9 &  9.4447e-13\\
          &  & &     &    &   &   &   \\
 \textbf{PJALM}   &1e-6      &1e-8 & 111   &1.67   & 2.4977e-6   &9.4450e-7 &  4.1335e-10\\
    &  &1e-14&   249 &  3.63 &  1.0173e-10  & 3.7225e-11 &  8.6506e-15\\
     &  & &     &    &   &   &   \\
  &1e-9      &1e-7 & 205   & 3.06   & 2.5369e-9   & 9.3210e-10 & 2.4510e-13\\
  &  &1e-15&  276  & 4.08  &  1.4143e-11  & 5.2002e-11 &  6.6543e-16\\
   &  & &     &    &   &   &   \\
\hline
 &1e-3      &1e-7 & 62   &0.85   & 4.8548e-5   & 1.7123e-5 & 9.3737e-8\\
          &  &1e-12& 176   & 2.64  &  2.7059e-9  & 9.7709e-10 &  9.1783e-13\\
           &  & &     &    &   &   &   \\
\textbf{HTY}  &1e-6      &1e-8 & 92   &  1.35 & 2.7184e-6   & 9.6661e-7 & 4.0385e-9\\
  &  &1e-14& 223   &3.15   & 1.1042e-10  & 4.0226e-11 &  9.3091e-15\\
   &  & &     &    &   &   &   \\
  &1e-9      &1e-7 & 176   & 2.78 & 2.7059e-9   & 9.7709e-10 & 9.1783e-13\\
  &  &1e-15&   243 &  3.70 & 2.8377e-11  & 1.0533e-11 &  4.4329e-16\\
   &  & &     &    &   &   &   \\
\hline
 &1e-3      &1e-7 & \emph{\textbf{61}}   &\emph{\textbf{0.82}}   & 7.4082e-5   & 3.3954e-5 & 1.5195e-9\\
          &  &1e-12& \emph{\textbf{127}}   & \emph{\textbf{1.84}}   & 5.8944e-8    &  1.6729e-7 & 1.3001e-13  \\
           &  & &     &    &   &   &   \\
\textbf{GR-PPA}  &1e-6      &1e-8 & 108   & 1.52   & 5.5130e-7    & 6.2676e-7  & 3.4315e-11  \\
  &     &1e-14     &  \emph{\textbf{172}}  & \emph{\textbf{2.56}}   & 2.9521e-10    &8.3742e-10   & 8.8742e-16  \\
   &  & &     &    &   &   &   \\
  &1e-9      &1e-7 & \emph{\textbf{167}}   & \emph{\textbf{2.42}}   &  5.3963e-10   &  7.3383e-10 &  3.7050e-14 \\
  &  &1e-15        & \emph{\textbf{172}}   & \emph{\textbf{2.41}}   & 2.9521e-10     & 8.3742e-10  & 8.8742e-16  \\
   &  & &     &    &   &   &   \\
\hline
  &1e-3      &1e-7 &  \emph{40}   &\emph{0.55}  &  9.8495e-5 & 1.2096e-4  &2.7440e-8  \\
        &  &1e-12& \emph{112} & \emph{1.53}  & 7.1036e-9  & 4.8224e-9 &  8.9763e-13   \\
           &  & &     &    &   &   &   \\
 \textbf{T-ADMM}   &1e-6      &1e-8 & \emph{72}    &\emph{1.02}  & 1.3128e-6  & 8.9570e-7  & 2.3510e-10 \\
  &     &1e-14     &\emph{147}     & \emph{2.12} &  7.4334e-11 &  5.0156e-11 & 9.7617e-15 \\
   &  & &     &    &   &   &   \\
  &1e-9      &1e-7 & \emph{125}    & \emph{1.70} & 1.3053e-9  &8.8746e-10   & 1.5974e-13 \\
  &  &1e-15        & \emph{160}    & \emph{2.01} & 1.3669e-11  &9.4374e-12   &6.6557e-16  \\
\hline
\end{tabular}}
\end{center}
\begin{center}
Table 3: Comparative results of different algorithms under different   tolerances.
\end{center}
}
\vskip4mm
\begin{itemize}
\item  The Proximal Jacobian Decomposition of ALM \cite{HeXuYuan2016} (denoted by ``PJALM'');
\item The splitting method in \cite{HeTaoYuan2015} (denoted by ``HTY'');
\item The generalized parametrized proximal point algorithm \cite{BaiLiLi2017} (denoted by ``GR-PPA'').
\item  The twisted version of the proximal ADMM \cite{WangSong2017} (denoted by ``T-ADMM'').
\end{itemize}

We set  $(\tau,s)=(0.9,1.09)$ for GS-ADMM-III  and the parameter $\beta = 0.05$
for all the comparison algorithms. The default parameter $\mu=2.01$ and $H = \beta \m{I}$ are used for
 HTY \cite{HeTaoYuan2015}. As suggested by the theory and numerical
experiments in \cite{HeXuYuan2016}, the proximal parameter is set as 2 for PJALM. As shown in  \cite{BaiLiLi2017}, the relaxation factor of GR-PPA is set as 1.8 and other default parameters are chosen as \[(\sigma_1,\sigma_2,\sigma_3,s,\tau,\varepsilon)=\left(0.178,0.178,0.178, 10,\frac{\sqrt{5}-1}{2},\frac{\sqrt{5}-1}{2}\right).\]
For T-ADMM, the symmetric matrices therein are chosen as $M_2=M_2=v\textbf{I}$ with $v=\beta$ and the correction factor is set as $a=1.6$ \cite{WangSong2017}.
The results obtained by the above algorithms under different tolerances are reported in Table 3.
With fixed tolerance $\textrm{TOL}=10^{-9}$ and $\textrm{Tol}=10^{-15}$,
the convergence behavior of the error measurements IER(k) and OER(k)
 by the five algorithms using different starting points are shown in Figs. 3-5.
From Table 3 and Figures 3-5, we may have the following observation:
\begin{itemize}
\item Under all different tolerances, GS-ADMM-III performs significantly
better than other four algorithms in both the number of iterations and CPU time.
\item GR-PPA is slightly better than PJALM and HTY, and T-ADMM outperforms  PJALM, HTY and GR-PPA.
\item the convergence curves  in Figs. 3-5  illustrate that using different starting points, GS-ADMM-III also converges
fastest among  the  comparing methods.
\end{itemize}
All these numerical results demonstrate the effectiveness and robustness of  GS-ADMM-III, which
is perhaps due to the symmetric updating of the Lagrange multipliers and the proper choice of the stepsizes.

\section{Conclusion}
Since the direct extension of ADMM in a Gauss-Seidel fashion
for solving the three-block separable convex optimization problem is not necessarily
convergent analyzed by Chen et al. \cite{Chen2016}, there has been a constantly increasing interest in developing and
improving the theory of the ADMM for solving the multi-block separable convex optimization.
In this paper, we propose an algorithm, called GS-ADMM, which could solve the
general model (\ref{Prob}) by taking advantages of the multi-block structure.
In our GS-ADMM, the Gauss-Seidel fashion is taken for updating
the two grouped variables, while the block variables within each group are updated
in a Jacobi scheme, which would make the algorithm more be attractive and effective for solving big
size problems. We provide a new convergence domain for the stepsizes of the dual
variables, which is significantly larger than the convergence domains given in the literature.
Global convergence as well as the $\C{O}(1/t)$ ergodic convergence
rate of the GS-ADMM is established.  In addition, two special cases of GS-ADMM, which allows
one of the proximal parameters to be zero, are also discussed.

This paper simplifies the analysis in \cite{HeMaYuan2016}
and provides an easy way to analyze the convergence of the symmetric ADMM.
Our preliminary numerical experiments show that with proper choice of parameters,
the performance of the GS-ADMM could be very promising.
Besides, from the presented convergence analysis, we can see that  the theories in the paper
can be naturally extended to use more general proximal terms, such as letting
$P_i^k(x_k) := \frac{\beta}{2}\|x_i - x_i^k\|_{\m{P}_i}$ and
$Q_j^k(y_j) := \frac{ \beta}{2}\|y_j - y_j^k\|_{\m{Q}_j}$
in (\ref{GS-ADMM}),  where $\m{P}_i$ and $\m{Q}_i$ are matrices such that
$\m{P}_i \succ (p-1) A_i \tr A_i$ and $\m{Q}_j \succ (q-1) B_j \tr B_j$
for all $i=1, \cdots, p$ and $j=1, \cdots, q$. Finally, the different ways of partitioning
the variables of the problem also gives the flexibility of GS-ADMM. \\

\begin{figure}[htbp]
 \begin{minipage}{1\textwidth}
 \def\figurename{\footnotesize Fig.}
 \centering
\resizebox{8cm}{8cm}{\includegraphics{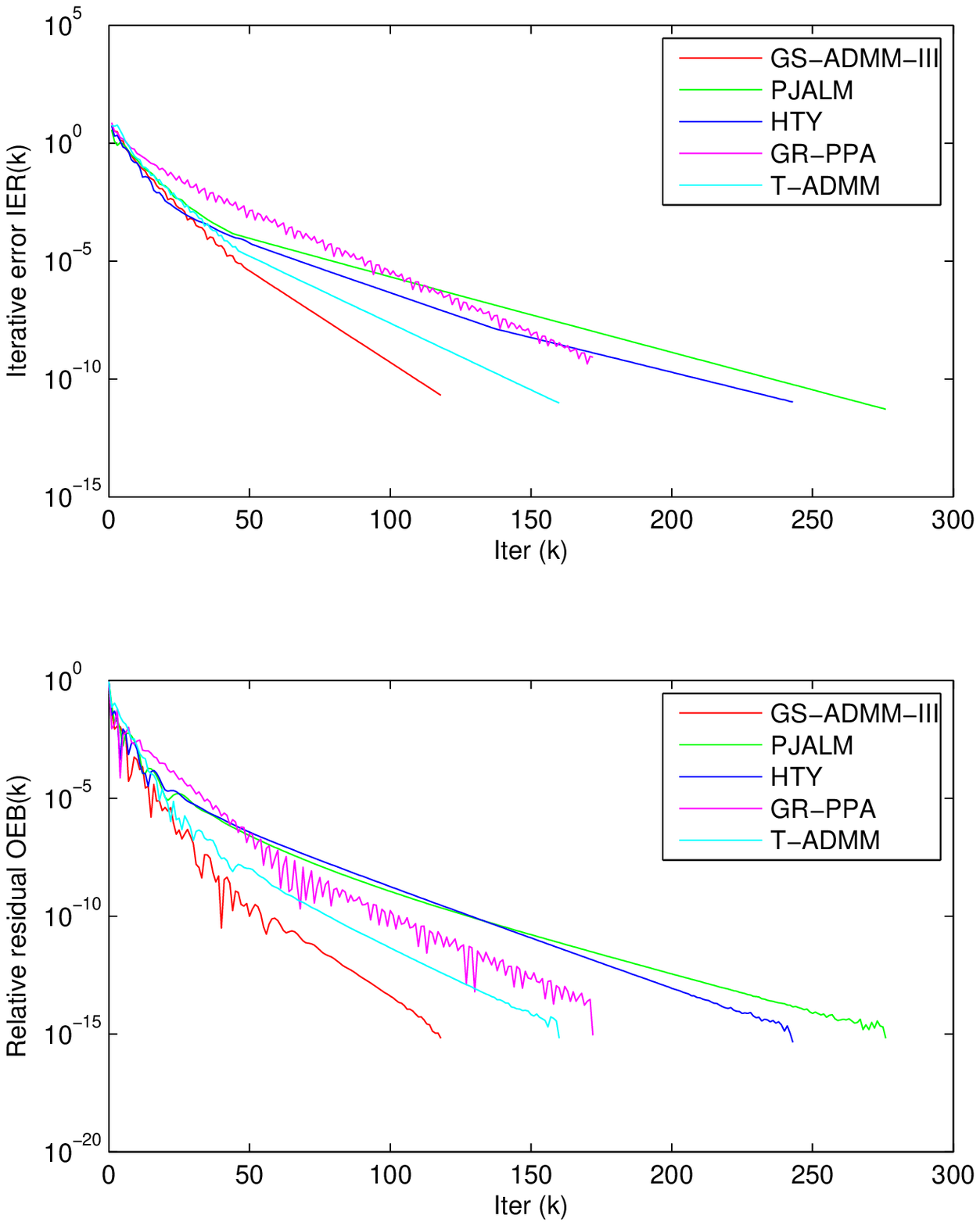}}
\caption{\footnotesize Convergence curves of IER and OER with initial values $(X^0,S^0, L^0,\Lambda^0)=(\textbf{I},2\textbf{I},\textbf{I},\textbf{0})$.}
\resizebox{8cm}{8cm}{\includegraphics{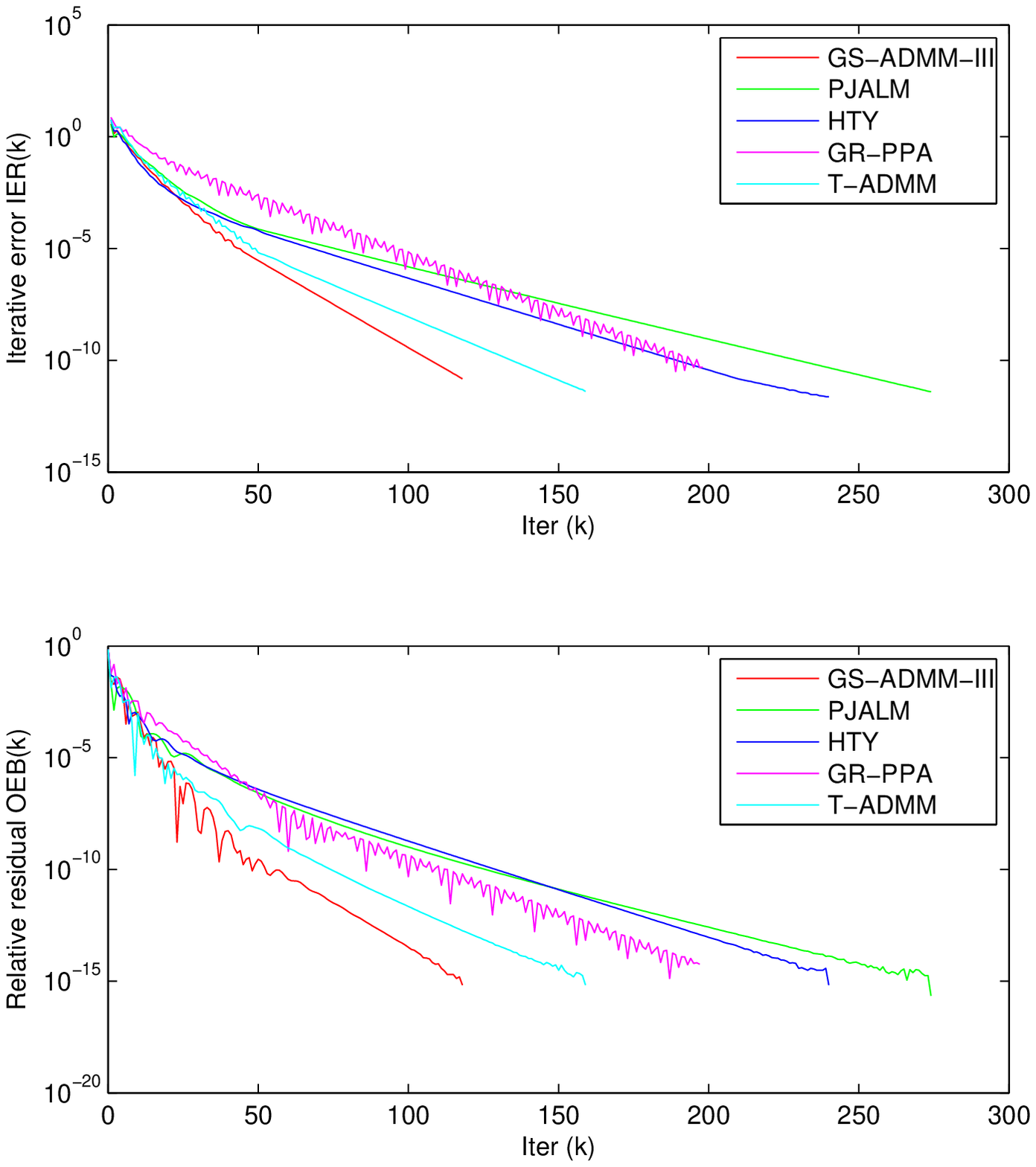}}
\caption{\footnotesize Convergence curves of IER and OER with initial values $(X^0,S^0, L^0,\Lambda^0)=(\textbf{I},\textbf{I},\textbf{0},\textbf{0})$.}
\resizebox{8cm}{8cm}{\includegraphics{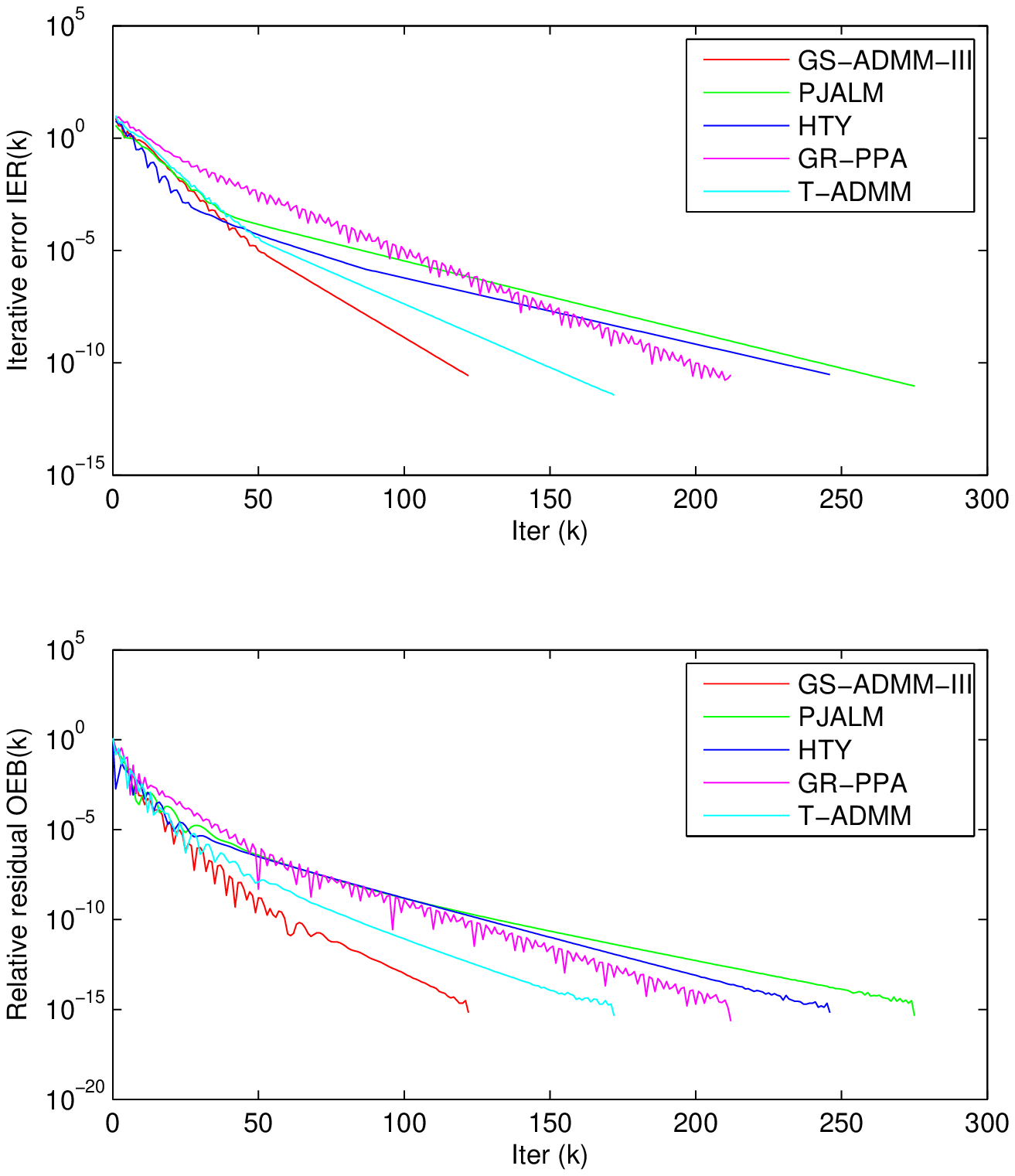}}
\caption{\footnotesize Convergence curves of IER and OER with initial values $(X^0,S^0, L^0,\Lambda^0)=(\textbf{I},4\textbf{I},3\textbf{I},\textbf{0})$.}
   \end{minipage}
\end{figure}

{\textbf{Acknowledgements.}}
The authors  would like to thank the anonymous referees for providing very constructive comments.
The first author also wish to thank Prof. Defeng Sun at National University of Singapore for his valuable discussions on ADMM
and Prof.  Pingfan Dai at Xi'an Jiaotong University for discussion on an early version of the paper.

\end{document}